\documentclass[oneside, book,12 pt]{article}

\usepackage{amsmath,amssymb,amsfonts,amsthm}

\usepackage{amsmath,amssymb,amsfonts,amsthm}
\usepackage{mathrsfs}

\usepackage{pstricks} 
\definecolor{darkblue}{rgb}{0 .2 .5}
\usepackage[colorlinks=true,linkcolor=darkblue,citecolor=darkblue,urlcolor=blue,pdfborder={0 0 0}]{hyperref}

\usepackage{enumerate}

\usepackage[font=sf, labelfont={sf,bf}, margin=1cm]{caption}
\usepackage{bbm}

\usepackage{graphicx}

\usepackage{cleveref}
  \crefname{theorem}{Theorem}{Theorems}
  \crefname{thm}{Theorem}{Theorems}
  \crefname{lemma}{Lemma}{Lemmas}
  \crefname{lem}{Lemma}{Lemmas}
  \crefname{remark}{Remark}{Remarks}
  \crefname{prop}{Proposition}{Propositions}
\crefname{question}{Question}{Questions}
  \crefname{defn}{Definition}{Definitions}
  \crefname{corollary}{Corollary}{Corollaries}
  \crefname{section}{Section}{Sections}
  \crefname{figure}{Figure}{Figures}

\newtheorem{thm}{Theorem}[section]
\newtheorem{lem}[thm]{Lemma}

\newtheorem{corollary}[thm]{Corollary}
\newtheorem{prop}[thm]{Proposition}
\newtheorem{defn}[thm]{Definition}

\newtheorem{question}[thm]{Question}

\numberwithin{equation}{section}

\newcommand{\END}{\textsf{END} }

\renewcommand{\P}{\mathbb P}

\newcommand{\E}{\mathbb E}
\newcommand{\R}{\mathbb R}
\newcommand{\N}{\mathbb N}
\renewcommand{\H}{\mathbb H}

\newcounter{mycount}
\newenvironment{mylist}{\begin{list}{{\rm (\roman{mycount})}}%
{\usecounter{mycount}\itemsep 0pt}}{\end{list}}

\title{Geometry and percolation on half planar triangulations}
 \author{Gourab Ray \footnote{University of British Columbia, 1984
     Mathematics Road, Vancouver, BC, V6T 1Z2}}
\date{{\small \today}}
\begin{document}

\maketitle

\begin{abstract}
  We analyze the geometry of domain Markov half planar triangulations. In \cite{AR13} it is
  shown that there exists a one-parameter family of measures supported on half planar triangulations
  satisfying translation invariance and domain Markov property. We study the geometry of these maps and show that they exhibit a sharp phase-transition in
  view of their geometry at $\alpha = 2/3$. For $\alpha<2/3$, the maps 
  form a tree-like stricture with infinitely many small cut-sets. For $\alpha > 2/3$,
  we obtain maps of hyperbolic nature with exponential volume growth and anchored expansion. Some results about the geometry of percolation clusters on such maps and random walk on them are also obtained.
\end{abstract}
\begin{figure}
 \centering{\includegraphics[scale =0.75 ]{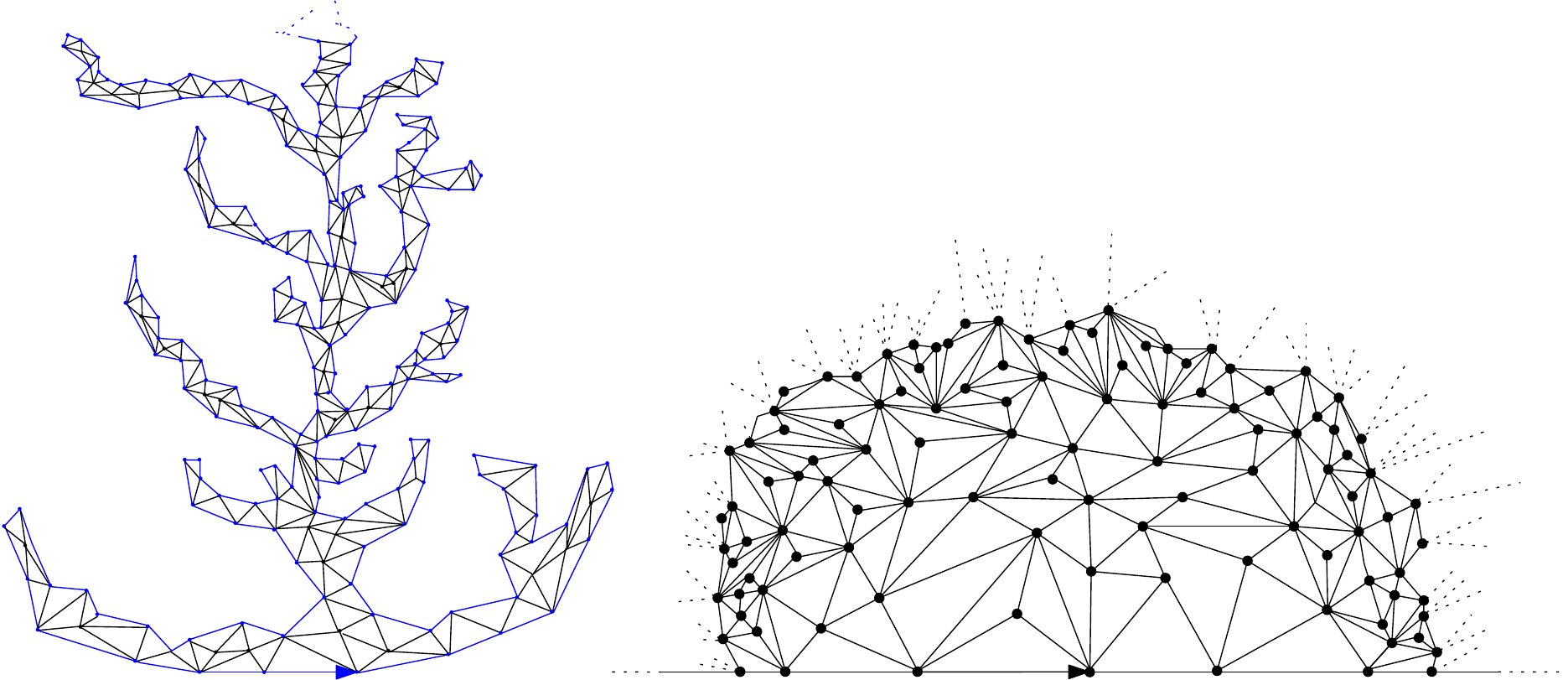}}
\caption{An illustration (artistic) of the geometry of a subcritical half planar triangulation to the left and that of supercritical to the right. The blue edges in the subcritical map is the boundary of the map.}\label{fig:sub_super}		 
\end{figure}

\section{Introduction}\label{sec:intro}
 Studying the geometry of random maps has been an area of major interest in the recent years (see \cite{UIPQinfty,UIPT2,BC11}). In \cite{AR13}, a classification Theorem for domain Markov half planar maps was proved and a phase-transition was observed in view of their geometry (details to follow). In this paper, we focus on the subcritical and supercritical phases of domain Markov half planar triangulations, and analyze this phase-transition in more detail. In particular, we obtain results for volume growth, isoperimetry and geometry of percolation clusters in the supercritical and subcritical phases of these maps. Finally, we extract some information about the behaviour of random walk on these maps from these geometrical informations. So this work can be viewed as a sequel to \cite{AR13}.

\begin{figure}[t]
\centering{\includegraphics[scale=1]{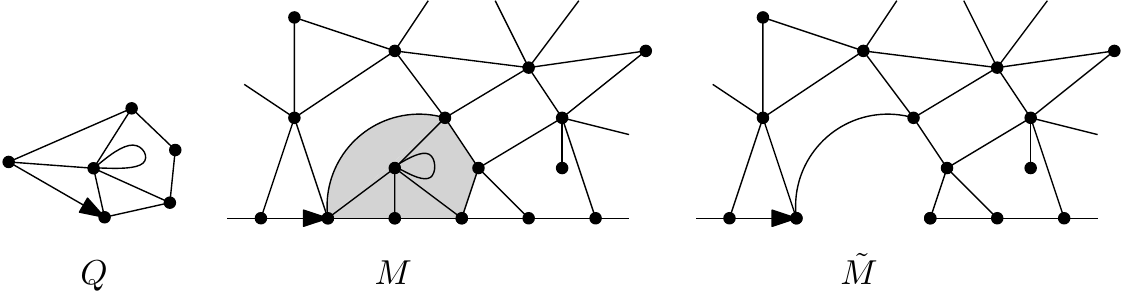} }
\caption{An illustration of domain Markov Property. Left: A finite simply connected map $Q$. 	Centre: A part of $M$ containing $Q$ with $2$ edges along the boundary. Right: The resulting map $\tilde{M}$ after removal of $Q$. Domain Markov property states that the law of $\tilde{M}$ is the same as that of $M$.}\label{fig: dmp}
\end{figure}

Recall that a \textbf{planar map} is a proper embedding of a connected (multi) graph on the sphere which is viewed up to orientation preserving homeomorphisms from the sphere to itself. For embeddings of infinite graphs, we assume that the graphs are locally finite (that is every vertex has finite degree) and the embedding is one-ended (the complement of any finite subset of the map has precisely one infinite connected component). By abuse of terminology, we shall identify the map with its (equivalence class of) embedding. Connected components of the complement of the embedding are called \textbf{faces}. The degree of a face is the number of edges incident to it. We focus on maps with a \textbf{boundary}, that is one face is marked as the external face and the edges and vertices incident to it form the boundary of the map. In this paper, the boundary will always be simple, that is, the boundary edges and vertices will form a simple cycle or an infinite simple path. In this paper, we focus on \textbf{half planar maps}, that is maps which are locally finite, one-ended and have an infinite simple boundary. In other words, these maps can be embedded in $\R \times \R_+$ with no accumulation points such that the boundary is $\R \times \{0\}$. More specifically, we focus on \textbf{half planar triangulations}, that is half planar maps where all the faces except the external face are triangles. All our maps are \textbf{rooted}, that is an oriented edge is specified as the root and in this paper the root is always on the boundary and is oriented in a way such that the external face is to the right of the root.

\section{Main results}\label{sec:main_results}
In \cite{AR13}, measures on half planar maps were considered which satisfy two natural properties: \textbf{translation invariance} and \textbf{domain Markov property}. Informally, we say a half planar random map is translation invariant if the law of the map is invariant with respect to translation of the root along the boundary. A sub-map of a half planar map is said to be simply connected if its union with the boundary is a simply connected subset of the plane. Roughly speaking, if we condition on a random half planar map $M$ to contain some simply connected subset with a simple boundary containing the root edge and remove it, and if the conditional distribution of the remaining map is the same as that of $M$, then we say that the law of $M$ satisfies the domain Markov property (see \cref{fig: dmp}). We refer the reader to \cite{AR13} for a more precise treatment.

Vertices not on the boundary of a half planar map are called \textbf{internal vertices.} We quote below a special case of the main result of \cite{AR13}.

\begin{thm}[\cite{AR13}]\label{thm:main}
  All translation invariant and domain Markov measures supported on half planar triangulations without self-loops form a one parameter family $\H_{\alpha}$ where the parameter $\alpha \in [0,1)$. Furthermore $\alpha$ denotes the probability of the event that the triangle adjacent to the root edge is incident to an internal vertex. 
\end{thm}
We remark here that the restriction to triangulations without self-loops is necessary to obtain a one-parameter family (see \cite{AR13}, Section 3.4 for more on this.)

The measure corresponding to $\alpha =2/3$ is the well-known \textbf{uniform infinite half planar triangulation} (UIHPT) (see \cite{UIPT2,AC13}). It is illustrated in \cite{AR13} that the measures $\H_\alpha$ must exhibit a phase-transition in view of their geometry at $\alpha =2/3$. The goal of this paper is to study the maps in the subcritical ($\alpha \in [0,2/3)$) and supercritical ($\alpha \in (2/3,1)$) regime of this one-parameter family.

\subsection{Geometry}
We present below the results obtained in this paper first for supercritical and then for subcritical maps. Roughly, the behaviour of supercritical maps are hyperbolic: they have exponential volume growth and anchored expansion. Anchored expansion is enough to guarantee that the simple random walk is transient. The subcritical maps behave, in view of their geometry, roughly like a critical Galton-Watson tree conditioned to survive (see \cite{Kestensubdiff}). They have quadratic volume growth and infinitely many cut-sets of finite size (see \cref{fig:sub_super}). All the terms stated in this paragraph will be defined rigorously below.

 We remark here that the geometric properties are certainly very different from the critical uniform infinite half planar triangulation (UIHPT). For results of similar nature regarding the UIHPT, see \cite{UIPT2,UIPT3,AC13}.

\subsubsection{Supercritical}
Roughly, the geometry of maps in the supercritical regime can be viewed as a collection of supercritical trees one attached to each vertex of an infinite simple path (see \cref{fig:heur}). Hence, we can expect exponential volume growth, large cut-sets and positive speed of random walk on these maps. The results which follow confirm some of these heuristics.

 Throughout this subsection, we assume $\alpha \in (2/3,1)$.
 For a set $X$, we write $|X|$ to denote its cardinality. By an abuse of notation, for any finite graph or map $G$, let $|G|$ denote its number of vertices. The \textbf{ball} of radius $r$ in a map denotes the submap formed by all the faces which have at least one vertex incident to it which is at a distance strictly less than $r$ from the root vertex along with all the edges and vertices incident to them. The \textbf{hull} of radius $r$ is the ball of radius $r$ along with all the finite components of its complement. Note that since the half planar maps are one-ended, there will be exactly one infinite component in the complement of the ball and the hull is always a simply connected sub-map. The \textbf{internal boundary} of a simply connected sub-map is the set of vertices and edges in the sub-map which is incident to at least one finite degree face which is not in the sub-map. Clearly, the internal boundary  of a hull is a connected simple path in the map. We denote the hull of radius $r$ around the root of a rooted map $M$ by $B_r(M)$ and the internal boundary of $B_r(M)$ by $\partial B_r(M)$. We sometimes will drop the $M$ and just write $B_r$ for the hull when the map in question is obvious.

We first show exponential volume growth of the hull and the boundary of the hull. 

\begin{thm}\label{thm:vol_h_triang}
Suppose $T$ is a map with law $\H_\alpha$ where $\alpha \in (2/3,1)$.
 There exists some constants $C>c>1$ such that almost surely,
\begin{equation}
  \limsup \frac{ |\partial B_r(T)| }{C^r}  =0 \quad \text{and } \quad  \liminf \frac{ |\partial B_r(T)| }{c^r} = \infty
\end{equation}
and also,
\begin{equation}
    \limsup \frac{ | B_r(T)| }{C^r}  =0 \quad \text{and } \quad  \liminf \frac{ |B_r(T)| }{c^r} = \infty
\end{equation}
\end{thm}

Having established the exponential volume growth, we ask if there are small cut-sets in the map. The usual parameter to look for in this situation is the Cheeger constant but since our maps are random and any finite configuration does occur almost surely somewhere in the map, the correct parameter to consider is the anchored expansion constant (see \cite{LP:book} Chapter 6).

For a graph $G$, let $V(G)$ denote its set of vertices. 	For any graph $G$, and a subset of vertices $S \subset V(G)$, let $|\partial_E S|$ denote the number of edges in $G$ with one vertex in $S$ and another in $V(G) \setminus S$. Also let $|S|_E$ denote the sum of the degrees of the vertices in $S$. The \textbf{anchored expansion constant} $i_E^*(G)$ of a graph $G$ is defined as 
\[
 i_E^*(G)= \liminf_{n \to \infty} \left\{ \frac{|\partial_E S|}{|S|_E} ; S \subset V(G)\text{ is connected, }v \in S, |S|_E \ge n\right\}
\]
 We say the graph $G$ has \textbf{anchored expansion} if $i_E^*(G)>0$.
Although we specify a vertex $v$ in the above definition, the definition is independent of the choice of $v$. 
\begin{figure}
 \centering{\includegraphics[scale =1 ]{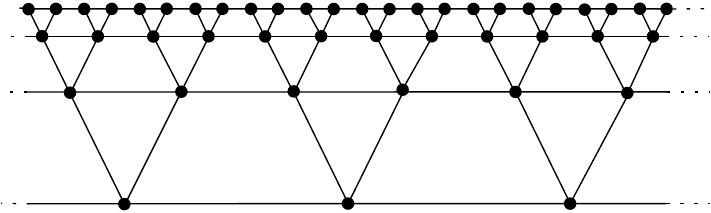}}
\caption{An very rough intuition of the geometry of supercritical maps.}\label{fig:heur}		 
\end{figure}
\begin{thm}\label{thm:anchored}
 A half planar triangulation with law $\H_\alpha$ for $\alpha \in (2/3,1)$ has anchored expansion almost surely.
 \end{thm}
 We remark here that the exponential lower bound for the volume growth can be concluded from anchored expansion, but we prove it using a different procedure involving an exploration process because we use the same exploration process to study the subcritical maps and also we get an upper bound on the volume growth using this method. 

A simple random walk on a random map is defined as follows: we fix a sample of the map and define a sequence $X_0,X_1,\ldots$ such that $X_0$ is the root vertex and after obtaining $X_i$, we choose uniformly one of the neighbouring edges of $X_i$ and define $X_{i+1}$ to be the vertex other than $X_i$ incident to that edge. It is shown in \cite{v00} that simple random walk on bounded degree graphs having anchored expansion has positive liminf speed. Unfortunately our maps are not bounded degree maps, so we cannot directly apply the result. However we can conclude using Theorem 3.5 of \cite{T92} and \cref{thm:anchored} that 
\begin{corollary}
 Simple random walk on a map with law $\H_\alpha$ is transient almost surely if $\alpha \in (2/3,1)$.
\end{corollary}
We believe that the random walk do have positive speed almost surely for supercritical maps. 
In fact we also believe that the distance of $X_n$ from the boundary of the map also grows linearly. We plan to take this up in a future paper.
  
\subsubsection{Subcritical}

Throughout this subsection, $\alpha \in [0,2/3)$.
The journey of understanding subcritical triangulations begins with a result about their cut-sets. A \textbf{cut-set} of an infinite rooted graph $G$ is a connected subgraph of $G$ which when removed breaks up $G$ into two or more connected components, the root being in the finite component. We shall see later (\cref{prop:cutset}) that in the subcritical regime there exists infinitely many cutsets each of which consists of a single edge almost surely.

Union of two graphs is the graph induced by the union of their vertices.  Presence of infinitely many cutsets consisting of a single edge (\cref{prop:cutset}) and domain Markov property entails that a triangulation $T$ distributed as a subcritical $\H_\alpha$ can be decomposed as 
\begin{equation}
       T = \cup_{i=1}^\infty T_i \label{eq:decomposition}
\end{equation}
where $T_i$'s are i.i.d. triangulations which are almost surely finite. Furthermore the decomposition is such that $T_i,T_j$ share a single edge if and only if $|i -j|=1$ (see \cref{fig:sub_super}). This can be for example proved using the peeling procedure described in \cref{sec:peeling0}. Such a decomposition also ensures that the subcritical triangulations are recurrent almost surely (\cref{prop:recurrent}). We believe that the spectral dimension of a subcritical map is almost surely $4/3$ because the subcritical maps fall in the family of strongly recurrent graphs as per \cite{kumagai2008heat}. See discussion in \cref{sec:open}.

 We can also consider the dual maps of these maps which consist of a vertex in each face and two vertices are joined together if their corresponding faces share an edge. As described in \cite{AR13}, we make such a dual map locally finite by breaking the infinite degree vertex corresponding to the infinite face into infinitely many leaves. The decomposition \eqref{eq:decomposition} entails that the dual of a subcritical triangulation almost surely consists of an i.i.d. sequence of finite graphs each of which contains vertices of degree either $3$ or $1$ and are connected to each other by a single edge.\ For $\alpha = 0$, it can actually be seen that the dual is a critical Galton-Watson tree conditioned to survive where the offspring distribution of the Galton-Watson tree is as follows: it produces two offsprings with probability $1/2$ and no offspring with probability $1/2$ (see \cite{AR13}). Hence the maps for different values of $\alpha$ can be seen as an ``interpolation'' between the UIHPT and critical trees. In fact, we believe that the scaling limit of such maps in the sense of local Gromov-Hausdorff topology exists and is the infinite non-compact CRT (which can be viewed as the tangent cone at the root of the compact Aldous CRT, (see \cite{BP13})).

The discussion above allows one to expect that the length of the boundary of the hull of radius $r$ is a tight sequence. We prove a stronger result: the boundary sizes of the hull has exponential tail.
\begin{thm}\label{thm: boundary_ball_tight}
Let $\alpha \in [0,2/3)$ and let $T$ be a map with law $\H_\alpha$. Then there exists some positive constant $c>0$ (depending only on $\alpha$) such that
\[
 \H_{\alpha}(|\partial {B_r(T)}|)  > n) < e^{-cn}
\]
 for all $n \ge 1$.
\end{thm}
The following central limit theorem shows that the volume growth is quadratic. This reconfirms the tree-like behaviour.

\begin{thm}\label{thm:hull_convergence}
Let $T$ is a map with law $\H_\alpha$ where $\alpha \in [0,2/3)$. Then 
\[
 \frac{|B_r(T)|}{r^2} \rightarrow S_{1/2}(\alpha)
\]
in distribution where $S_{1/2}(\alpha)$ is a stable random variable with parameter $1/2$ where the other parameters of $S_{1/2}(\alpha)$ depends only upon $\alpha$.
\end{thm}

\subsection{Percolation}  \label{sec:percolation}
The interest in studying percolation on random maps stems from the need to understand the connection between statistical physics models on random surfaces and the Euclidean plane. Bernoulli percolation is the simplest of such models. We are mainly interested in quenched statements about Bernoulli site percolation on random triangulations: take a half planar triangulation $T$ with law $\H_\alpha$ and color each vertex independently black with probability $p$ or white with probability $1-p$. A black (resp. white) cluster is a connected component induced by the black (resp. white) vertices on the map. Given a half planar map,	 denote the percolation measure on it by $P_p$ and the expectation by $E_p$. Let $\P_p$ denote the overall measure of percolation configuration on a random map with law $\H_\alpha$ and let $\E_p$  denote the expectation with respect to the measure $\P_p$. It is understood that in these notations there is a hidden parameter $\alpha$ which we shall drop to lighten notation. As usual, define $p_c$ to be the infimum over $p$ such that there exists an infinite black cluster $P_p$-almost surely. Further, we are also interested in 
\[
p_u = \inf\{p\in (0,1]: \text{ there exists a unique infinite cluster $P_p$-almost surely}\}
\]

 Percolation on random maps has been an object of interest for some time \cite{UIPT2,UIPT3,AC13}. For example, it  is shown (see \cite{UIPT3}) that $p_c =p_u = 1/2$ almost surely for site percolation on the uniform infinite half planar triangulation and almost surely clusters are finite at criticality. Geometry of the critical cluster in the UIPT is studied in \cite{CK13}. We want to understand how the behaviour of percolation clusters change if we step away from the critical regime. However it is immediate to see via \cref{prop:cutset} that percolation is uninteresting in the subcritical maps (in this case $p_c=1$ almost surely.) It was conjectured (see \cite{bperc96}) by Benjamini and Schramm that on non-amenable quasitransitive graphs, $p_c < p_u$. For supercritical maps, because of anchored expansion as depicted by \cref{thm:anchored}, we would expect a similar behaviour. 
 \begin{thm}\label{thm:site}
Fix $\alpha \in (2/3,1)$. Then $\H_\alpha$-almost surely,
\begin{mylist}
 \item $p_c = \frac{1}{2}\left(1-\sqrt{3-\frac{2}{\alpha}} \right)$
 \item $p_u = \frac{1}{2}\left(1+\sqrt{3-\frac{2}{\alpha}} \right).$
\end{mylist}
Also $\H_\alpha$-almost surely, there is no infinite black cluster $P_{p_c}$-almost surely and there is an unique infinite black cluster $P_{p_u}$-almost surely.
\end{thm}
Note that $p_c<p_u$ almost surely in the regime $\alpha \in  (2/3,1)$. It is interesting to note that as $\alpha \to 2/3$, both $p_c \to 1/2$ and $p_u \to 1/2$. But in the regime $(p_c,p_u)$ we have more than one infinite cluster. One can easily conclude via ergodicity of these maps with respect to translation of the root along the boundary (see \cite{AR13}, Proposition 1.3) that the number of infinite black or white clusters is actually infinite almost surely. The next Theorem shows that the number of black or white infinite cluster touching the boundary has positive density along the boundary.

Distance between two boundary vertices \textbf{along the boundary} is the number of edges on the boundary between them.
 Let $W_k^\infty,B_k^\infty$ be the number of infinite white and black clusters respectively which share at least a vertex which is within distance $k$ from the root along the boundary.
 \begin{thm}\label{thm:density}
 Fix $\alpha \in (2/3,1)$ and suppose $p \in (p_c,p_u)$ where $p_c$,$p_u$ are as in \cref{thm:site}. There exists a positive constant $\rho > 0$ such that almost surely,
\begin{equation}
 \frac{W_k^\infty}{k} \to \rho,  \hspace{5mm} \frac{B_k^\infty}{k} \to \rho
\end{equation}
\end{thm}
 The constant $\rho$ is in fact half of the probability of the event of having an infinite interface starting from a boundary edge (see \cref{sec:percolation} for more details.) 

A \textbf{ray} in an infinite percolation cluster is a semi-infinite
simple path in the cluster starting from a vertex closest to the root (with ties broken arbitrarily). Two rays $r_1$ and $r_2$ are equivalent if there is another ray $r_3$ which intersect both $r_1$ and $r_2$ infinitely many times. An \textbf{end} of a cluster is an equivalence class of rays. Let $\END(\mathcal C)$ denote the space of ends of a percolation cluster $\mathcal C$. We shall define a metric on $\END(\mathcal C)$ as follows: for any two rays $\xi$ and $\eta$ on $\mathcal C$, define the distance between them as
\begin{multline} 
 d(\xi,\eta) = \inf\{1/n, n=1 \text{ or } \forall X \in \xi, \forall Z \in \eta, \exists \text{ a component $K$ of }\\ \mathcal C \setminus B_n,
 |X \setminus K| + |Z \setminus K| < \infty\}\nonumber
\end{multline}
It is easy to deduce that $\END(\mathcal C)$ does not depend on the choice of the vertex around which we consider the graph-distance balls and that $\END(\mathcal C)$ equipped with this metric is compact.

\begin{thm}\label{thm:ends}
Fix $\alpha \in  (2/3,1)$. Assume $p_c,p_u$ are as in \cref{thm:site} and fix $p \in (p_c,p_u)$. Then $\H_\alpha$-almost surely,
 the subgraph formed by each infinite cluster has no isolated end and has continuum many ends $P_p$-almost surely. 
\end{thm}

Also for two sequences $\{a_n\}$ and $\{b_n\}$, $a_n \sim b_n$ means that $a_n/b_n$ converges to $1$ as $n \rightarrow \infty$. Also $a_n = O(b_n)$ means that there exists a  contant $M > 0$ independent of $n$ such that $|a_n|<M|b_n|$. For any sequence $\{a_n\}_{n\ge 1}$, $\Delta a_n:=a_{n+1} - a_n$. Further, the positive constant $c$ might change from one line to the next, but we still denote them by $c$ for simplicity.

\paragraph{Acknowledgement:} The author is grateful to Omer Angel for several illuminating discussions throughout the course of the research. The author is also indebted to Nicolas Curien for several insightful conversations during the authors visit to Saint-Flour summer school in 2011. The author also thanks the anonymous referee for a careful reading of the manuscript.

\tableofcontents

\section{Background}\label{sec:notation}

\begin{figure}
 \centering{\includegraphics[scale=.7]{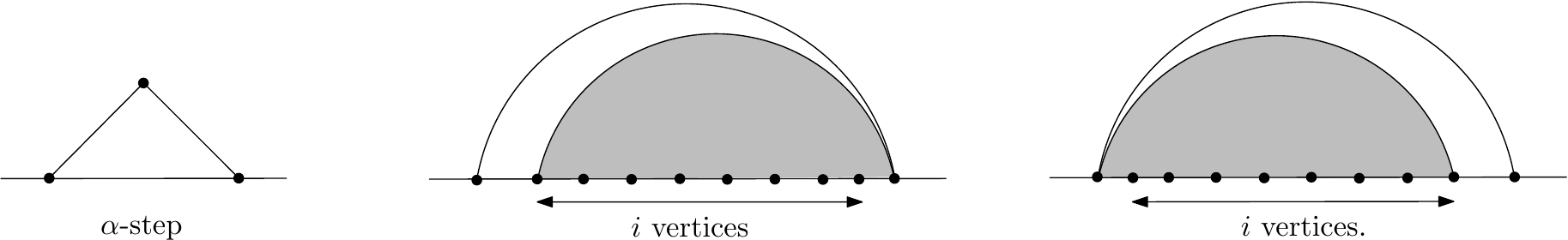}}
\caption{Left: An $\alpha$-step. Centre: A step of the form $(R,i)$. Right: A step of the form $(L,i)$. The gray area denotes some unspecified triangulation.} \label{fig:forms} 
  \end{figure}
The goal of this section is to review in more detail the phase transition observed in \cite{AR13} and describe   the process of peeling in \cref{sec:peeling0} which will play a central role throughout this paper. Also, we collect some preliminary results which we shall need later.

We first define certain events which will be used repeatedly in what follows. Following the notation of \cite{AR13}, an $\alpha$-step is the event in which the third vertex of the triangle incident to the root edge is an internal vertex. A step of the form $(L,i)$ (resp. $(R,i)$) is the event that the triangle incident to the root edge is attached to a vertex on the boundary which is at a distance $i$ to the left (resp. right) of the root edge along the boundary (see \cref{fig:forms}). We shall also talk about such events with the root edge replaced any fixed edge on the boundary of the map. Because of translation invariance, the measures of such events do not depend on the edge we want to consider and it was also shown in \cite{AR13} that for any fixed $i \ge 1$, the measures of $(L,i)$ and $(R,i)$ are the same. Let $p_{i,k}$ denote the measure of the event that a step of the form $(L,i)$ or $(R,i)$ occurs and the triangle incident to the root edge separates $k$ internal vertices of the map from infinity. Let $p_i = \sum_kp_{i,k}$. We record here some computations done in \cite{AR13}. Following the notation of \cite{AR13}, we denote by $\beta$ the probability of the event of the form $(R,1)$ with no internal vertex in the $2$-gon enclosed by the triangle incident to the root edge. 

\begin{description}
 \item [Subcritical $\alpha <2/3$:] 
\begin{align} 
\beta & = \frac{(2-\alpha)^2}{16} \label{eq:betasub}\\
p_{i,k} &= \frac{2}{4^{i}}\phi_{k,i+1}\left(1-\frac{\alpha}{2}\right)^{2i}\left(\frac{\alpha}{4} \left(1-\frac{\alpha}{2}\right)^2\right	)^k\label{eq:sub_pik}\\
 p_i &= \frac{2}{4^i}\frac{(2i-2)!}{(i-1)!(i+1)!} ((1-3\alpha/2) i+1)\label{eq: sub_pi}
\end{align}

\item[Supercritical $\alpha > 2/3$:] 
\begin{align} 
\beta & = \frac{\alpha(1-\alpha)}{2}\label{eq:betasuper}\\
p_{i,k} &= 2 \phi_{k,i+1}\alpha^{i+2k}\left(\frac{1-\alpha}{2}\right)^{i+k}\label{eq:super_pik}\\
 p_i &= \frac{2}{4^i}\frac{(2i-2)!}{(i-1)!(i+1)!}\left(\frac{2}{\alpha}-2\right)^{i}((3\alpha-2)i + 1) \label{eq: super_pi}
\end{align}
\end{description}
One can verify using Stirling's formula, that the asymptotics of $p_i$ are as follows:
\begin{itemize}
 \item One can easily compute using Stirling's formula that
\begin{equation}
 p_i \sim \frac{1-3\alpha/2}{2\sqrt{\pi}} i^{-3/2} \label{eq:constant1}
\end{equation}
So the tail of $p_i$ is heavy with infinite expectation. 

\item $p_i\sim c \left(\frac{2}{\alpha}-2\right)^{i} i^{-3/2}$ for supercritical $\H_{\alpha}$ for some constant $c>0$. So $p_i$ has exponential tail.
\end{itemize}

\medskip
\subsection{Peeling}\label{sec:peeling0}
In this section we shall describe the concept of peeling which is the central tool used in this paper. Peeling has its roots in the
physics literature \cite{Wat,Amb}, and was used in the present form in
\cite{UIPT2}.\ It is useful for analyzing many aspects of planar maps which include percolation, random walks, volume growth and conformal properties (see \cite{UIPT2,AC13,curien13glimpse,menard2013perc}). In this paper we will use this procedure to analyze the geometry of domain Markov half planar triangulations. Let us remark here that this procedure can be used to analyze not only triangulations, but also other classes of random maps (see \cite{AR13,AC13,BC11}).

Suppose we have a sample $T$ from $\H_\alpha$ for some $\alpha \in (0,1]$. We will construct a growing sequence of simply connected sub-maps $P_n$ with a simple boundary and containing the root. We define $T_n$ to be the set of finite degree faces in $T$ not in $P_n$ along with the edges and vertices incident to them (note that $T_n$ is also a half planar triangulation because $P_n$ is simply connected with a simple boundary.) We will sometimes refer to $T_n$ as the complement of $P_n$. Start with $P_0$ to be empty and $T_0 = T$. At the $n$th peeling step, we pick an edge $e_n$ on the boundary of $T_n$ and add the triangle in $T_n$ incident to $e_n$ along with the finite component of the complement, if there is any (note that there can be at most one such component), to $P_{n+1}$. Define $T_{n+1}$ to be the complement of $P_{n+1}$. The root of $P_n$ is the root of $T$ and the root of $T_{n+1}$ is defined as the leftmost edge in the boundary of $P_n$ which is incident to at least one face of $T_n$ and is oriented from left to right. Note that at every step, the choice of $e_n$ does not depend on $P_n$ and for any such choice, $T_n$ is independent of $P_n$ and is distributed as $\H_{\alpha}$ via the domain Markov property. By abuse of notation, sometimes we shall re-root $T_n$ on some other edge on the boundary of $T_n$ and the distribution of $T_n$ does not change by translation invariance.

Notice that in any step of peeling on an edge $e$, there are essentially three possible choices for the peeling steps. Either it is an $\alpha$-step where the third vertex of the triangle incident to $e$ is an internal vertex of the unexplored part. Such a step has probability $\alpha$. Otherwise, it is of the form $(L,i)$ or $(R,i)$ (see \cref{fig:forms}) for some $i\ge 1$ and such an event occurs with probability $p_i$. Notice that when an event of the form $(L,i)$ or $(R,i)$ occurs we divide the whole triangulation into a finite and an infinite component. Conditioned on the triangulation in the finite component as well as all the triangulation revealed so far, the infinite component is again distributed as $\H_\alpha$ by the domain Markov property. The distribution of the triangulation of the finite component can be easily computed using \cref{eq:super_pik,eq:sub_pik}. On the event $(R,i)$ ( or $(L,i)$), the distribution of the triangulation in the finite component is called free triangulation of the $(i+1)$-gon with parameter $\alpha\beta$. Details about free triangulations along with computations of some estimates on them is done in \cref{sec:free_triangulation}.

 \begin{figure}[t]

\centering{\includegraphics[scale=0.75]{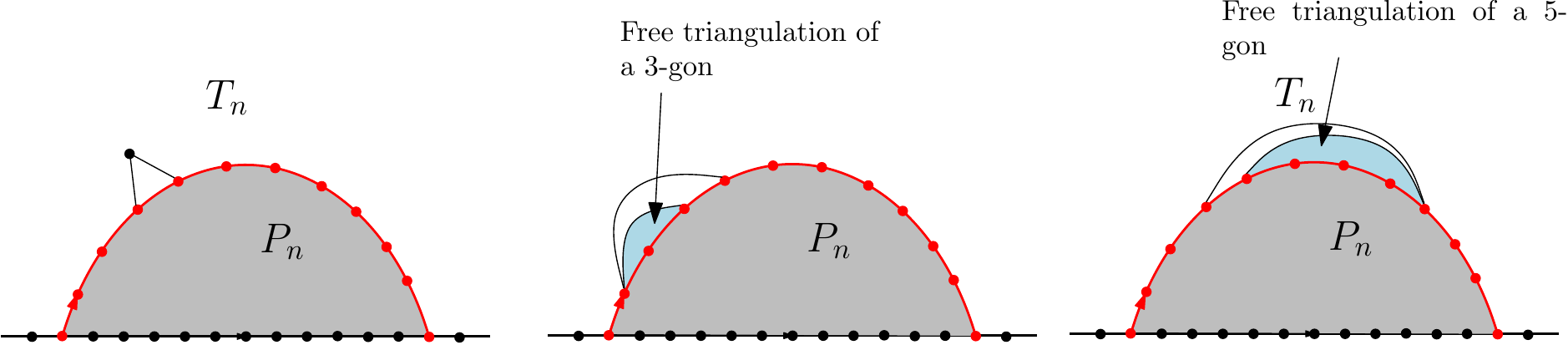} }
\caption{An illustration of possible peeling steps after $n$ steps of peeling have been completed. The gray area denote the peeled part $P_n$. The red vertices and edges denote the internal boundary of $P_n$. In the middle (resp. right), the blue regions correspond to a triangulation distributed as a free triangulation of a triangle (resp. pentagon). \label{fig:peeling_algo}}
\end{figure}

Intuitively, since $p_i$ has a heavy tail ($p_i \sim i^{-3/2}$) in the subcritical regime, steps of the form $(R,i)$ or $(L,i)$ occurs for large $i$ more frequently in the subcritical regime when performing the peeling process. This eats up the boundary a lot more resulting in the tightness of the peeling boundary of $P_n$. In contrast, the supercritical regime has exponential tail for $p_i$ ($p_i \sim \exp(-ci) , c>0$) which results in an exponentially growing boundary.


\subsection{Enumeration of planar maps}\label{sec:counting_formulas}
A \textbf{triangulation of an $m$-gon} is a finite map in which all faces are triangles except an external face of degree $m$. The boundary of the map must form a simple loop of $m$ edges. The root is on the boundary oriented such that the external face is to its right. The following combinatorial result may be found in
\cite{GFBook}. It is derived using the techniques introduced by Tutte
\cite{Tutte}.

\begin{prop}\label{prop:count}
  For $n,m\geq 0$, not both 0, the number of rooted triangulations
  of a disc with $m+2$ boundary vertices and $n$ internal vertices without self loops is
  \[
  \phi_{n,m+2} = \frac {2^{n+1} (2m+1)! (2m+3n)!} {m!^2 n! (2m+2n+2)! } .
  \]
\end{prop}

We will assume $\phi_{0,2}=1$ (see discussion in \cite{AR13}). The following asymptotics of $\phi_{n,m}$ may be found in the proof of Lemma 4.1 in \cite{AR13}. As $\{m,n\} \to \{\infty,\infty\}$,
\begin{equation}
\phi_{n,m}\sim  c \left(\frac{27}{2} \right)^n 
      9^m n^{-5/2} \sqrt{m} \left(1+\frac{2m}{3n}\right)^{2m+3n}
    \left(1+\frac{m}{n}\right)^{-2m-2n} \label{eq:asymp}
\end{equation}
for some constant $c>0$.


\subsection{Free triangulations }\label{sec:free_triangulation}
The following
measure is of particular interest:

\begin{defn}\label{def:free}
  The \textbf{Boltzmann} or \textbf{free} distribution on rooted triangulations
  of an $m$-gon with parameter $q\leq\frac{2}{27}$ is the probability
  measure that assigns weight $q^n / Z_m(q)$ to each rooted
  triangulation of the $m$-gon having $n$ internal vertices, where
  \[
  Z_m(q) = \sum_n \phi_{n,m} q^n .
  \]
\end{defn}

A freely distributed triangulation with parameter $q$ of an $m$-gon
will be referred to as a \textbf{free triangulation} with parameter $q$ of
an $m$-gon. 
Note that by the asymptotics of $\phi$ as $n \to \infty$ we see that the
sum defining $Z_m(q)$ converges for any $q \leq \frac{2}{27}$ and for no larger
$q$. The precise value of the partition function will be useful, and we
record it here:

\begin{prop}\label{prop:Z}
  If $q = \theta(1-2\theta)^2$ with $\theta \in [0,1/6]$, then
  \[
  Z_{m+2}(q) = ((1-6\theta) m +2-6\theta) \frac{(2m)!}{m!(m+2)!} 
  (1-2\theta)^{-(2m+2)}.
  \]
\end{prop}

The proof can be found as intermediate steps in the derivation of
$\phi_{n,m}$ in \cite{GFBook}. The above form may be deduced after a
suitable reparametrization of the form given there.
Let $I_m(q)$ denote the number of internal vertices of a freely distributed triangulation of an
  $m$-gon with parameter $q$. 

\begin{prop}\label{prop:free_expectation}
  Fix $\theta \in [0, 1/6)$ and let $q=\theta(1-2\theta)^2$. Fix an integer $m \ge 2$.
  \begin{mylist}
   \item $\E(I_m(q)) = \frac{(m-1)(2m-3)2\theta}{(1-6\theta)m + 6\theta} = \frac{4\theta}{(1-6\theta)}m + O(1)$
\item  $Var(I_m(q)) = \frac{(m-1)(2m-3)m(1-2\theta)}{((1-6\theta)m+6\theta)^2(1-6\theta)} = \frac{2(1-2\theta)}{(1-6\theta)^3}m + O(1)$
  \end{mylist}
\end{prop}

\begin{proof}
Note the following identity
\begin{equation}
 \E(I_{m}(q)) = \frac{q(Z_m'(q))}{Z_{m}(q)} = q (\log Z_m)'(q)\nonumber
\end{equation}
Putting $q = \theta(1-2\theta)^2$ and using \cref{prop:Z}, we obtain after an easy computation
\begin{equation}
  \E(I_{m}(q)) = \frac{(m-1)(2m-3)2\theta}{(1-6\theta)m + 6\theta}\nonumber
\end{equation}
The proof of (ii) is a similar computation and is left to the  reader to verify.
\end{proof}
We will need the following estimates whose proof is postponed to \cref{sec:tailproof}.
\begin{lem}\label{lem:tail}
Suppose $\alpha\beta = \theta (1-2\theta)^2$ where $\alpha \in [0,2/3)  $ and $\beta$ is given by \eqref{eq:betasub}. Suppose $Y$ is a variable supported on $\N \cup \{0\}$ such that $\P(Y=i) =p_i$ and $\P(Y=0) =\alpha$ where $p_i$ are given by \eqref{eq: sub_pi}. Then  
\begin{mylist}
  \item $\P(Y+I_{Y+1} > x) \sim \frac{c_\alpha }{\sqrt{x}}$
\item $\E(Y+I_{Y+1})\mathbbm{1}_{\{Y+I_{Y+1}<x\}} \sim c_\alpha \sqrt{x}$
\end{mylist}                                                                                                                                                                                                                                                                                                                                                                                                                                                                                                                                                                                                as $x \to \infty$ where 
\begin{equation}
         c_\alpha = \frac{(1-3\alpha/2)\sqrt{1-2\theta}}{\sqrt{\pi(1-6\theta)}}\label{eq:const}
\end{equation}
\end{lem}


\subsection{Stable Random Variables}
The theory of stable random variables plays a vital role in our subsequent analysis. Fix $\alpha \in (0,2]$. An independent sequence $X_1,X_2,\ldots$ is said to follow a stable distribution of type $\alpha$ if $S_n = X_1+\ldots +X_n$ satisfies
\[
 S_n \stackrel{(d)}= n^{1/\alpha}X_n +\gamma_n
\]
for some sequence $\gamma_n$ and the distribution of $X_1$ is not concentrated around $0$. See for example \cite{fellerII} Chapter VI or \cite{Durbook} for more details.

We shall be needing the following classical result. This can be found in \cite{Durbook}.
\begin{thm}\label{thm: stable}
 Suppose $X_1,X_2, \ldots$ are i.i.d. with a distribution that satisfies
\begin{enumerate}
 \item $\lim_{x \rightarrow \infty} \P(X_1>x)/ \P(|X_1| > x) = \theta \in [0,1]$
 \item $\P(|X_1| > x) = x^{-\alpha}L(x)$
\end{enumerate}
where $\alpha <2$ and $L$ is slowly varying. Let $S_n = X_1+\ldots X_n$. $a_n = \inf\{x: \P(|X_1| > x) \le n^{-1}\}$ and $b_n = n\E(X_1 1_{|X_1| \le a_n})$. As $n \rightarrow \infty$ $(S_n - b_n)/a_n \rightarrow Y$ in distribution where $Y$ is a stable random variable of type $\alpha$.
\end{thm}

We are specially interested in the case $\alpha=1/2$. It turns out that we can add a constant to a variable following a stable distribution of type $\alpha$ where $\alpha \neq1$ such that $\gamma_n = 0$ for all $n$ in its definition (see \cite{fellerII}). After such a centering, its density can be explicitly written as
\[
 (2\pi x^3)^{-1/2} \exp(-1/2x) \mathbbm1_{ \{x>0\}}
\]
This is known as the \textit{L\'{e}vy distribution}.



\section{Geometry}\label{sec:growth}
\begin{subsection}{Peeling algorithm}\label{sec:peel_algo}
Recall from the discussion in \cref{sec:peeling0} that we are free to choose the edge $e_n$ on which we apply the $n$th peeling step. We now describe an algorithmic procedure to choose the edges in such a way that at a certain (random) step we reveal the hull of the ball of radius $r$ around the root vertex. The algorithm follows the idea developed in \cite{UIPT2} for analyzing the volume growth of the full plane UIPT, but we modify it appropriately for the half plane versions. We take up the notations of \cref{sec:peeling0}. Further recall that the hull of the ball of radius $r$ of a map $M$ around the root is denoted by $B_r(M)$. 


 Suppose we perform the peeling procedure on a half planar triangulation $T$. 
Let $\tau_0=0$ and let $P_0$ be the root vertex. Suppose we have defined a (random) time $\tau_r$ so that $P_{\tau_r} = {B_r(T)}$ for some $r \ge 1$. In particular, the internal boundary of $P_{\tau_r}$ is $\partial {B_r(T)}$. The idea is to iteratively peel the edges in $\partial B_r(T)$ till none of the vertices in $\partial {B_r(T)}$ remain in the boundary of $T_n$.

\begin{description}
\item \textbf{Algorithm: }Suppose we have described the process up to step $n$ such that $\tau_r \le n < \tau_{r+1}$. Now look for the left most vertex $v$ of $\partial {B_r(T)}$ which remains in the internal boundary of $P_n$ at step $n$ and perform a peeling step on the edge to the right of $v$ in the boundary of $T_n$. If there is no vertex $v$ of $\partial B_r(T)$ left in the boundary of $T_{n+1}$, define $n+1 = \tau_{r+1}$ and $P_{\tau_{r+1}} = {B_{r+1}}$. 
\end{description}
The algorithm proceeds in such a way that for every vertex of $\partial B_r(T)$, we keep on peeling at an edge incident to that vertex until it goes inside the revealed map. Hence at step $\tau_{r}$, we reveal nothing but the hull of the ball of radius $r$ for every $r \ge 1$. Recall that the internal boundary of $P_n$ is the set of edges and vertices which are incident to at least one finite degree face not in $P_n$. Let $X_n$ denote the number of vertices in the internal boundary of $P_n$ at the $n$th step. It is easy to see that $X_n$ itself is not a Markov chain because the transition probabilities very much depend on the position of the edge on which we are peeling. However, a bit of thought reveals that $X_{\tau_r}$ is in fact an irreducible aperiodic Markov chain. We record the above observations in the following Proposition.

\begin{prop}\label{prop:algorithm}
For $r \ge 1$, $P_{\tau_r}$ described in the algorithm above is the same as $B_r(T)$ and $X_{\tau_r} = |\partial B_r(T)|$. Also, the sequence $\{X_{\tau_r}\}_{r \ge 1}$ is an irreducible aperiodic Markov chain. 
\end{prop}

 Following the idea of \cite{UIPT2}, we estimate the size of the boundary by analyzing $X_n$ separately for $\alpha$ in subcritical and supercritical regimes. Observe that $\Delta X_n = X_{n+1}-X_n \le 1$ for any $n$. Note also that the tails of $\Delta X_n$ have different behaviour in the subcritical and supercritical regimes:
\begin{equation}
\H_{\alpha}(\Delta X_n <- i) \approx
\begin{cases}
i^{-1/2} & \alpha<2/3\\
\exp(-ci) & \alpha>2/3
\end{cases}
\end{equation}
for some constant $c>0$ and a large non negative integer $i<X_n$. Thus, $\Delta X_n$ conditioned on $X_n$ has negative expectation if $X_n$ is not too small in the subcritical regime. This tells us that $X_n$ has a drift towards $0$ as soon as it gets large which implies it should be a tight sequence. On the other hand, it will follow from the computation below (\cref{lem:expected_change}) that in the supercritical regime, $\Delta X_n$ conditioned on $X_n$ has positive expectation. This will imply that $X_n$ grows linearly. This constitutes the key point of difference between the two regimes which is made rigorous in the following \cref{sec:supercritical,sec:subcritical}.

\end{subsection}

\subsection{Supercritical}\label{sec:supercritical}
In this subsection, we prove \cref{thm:vol_h_triang} and hence we assume $\alpha>2/3$ throughout this subsection. Recall that we denote by $B_r$ the hull of the ball of radius $r$ of a map $T$ with law $\H_\alpha$. Further, we shall also borrow the notations from \cref{sec:peeling0,sec:peel_algo}.

 As mentioned before, we will perform the peeling algorithm described in \cref{sec:peel_algo} and analyze the quantity $\Delta X_n$. To that end, we shall approximate $\Delta X_n$ by a sequence of auxilary variables $\tilde{X_n}$ such that the variables $\Delta \tilde{X_n}= \tilde{X_{n+1}} - \tilde{X_n}$ for $n\ge 1$ form an i.i.d. sequence with $\Delta \tilde{X_n} = -i$ if a step of the form $(L,i)$ or $(R,i)$ occurs in the $(n+1)$th peeling step and $\Delta \tilde{X_n}=1$ if an $\alpha$ step occurs in the $(n+1)th$ peeling step. Clearly, from definition, $X_n>\tilde{X_n}$ since if in a peeling step the triangle revealed has the third vertex not on the internal boundary of $P_n$, $\Delta X_n > \Delta \tilde{X_n}$.

Because of the exponential tail, the variables $\Delta \tilde{X_n}$ in the supercritical regime have finite variance. Further its expectation turns out to be positive.

\begin{lem}\label{lem:expected_change}
\begin{equation}
 \E(\Delta(\tilde{X_n})) = \sqrt{3\alpha - 2}{\sqrt{\alpha}}\label{eq: expected_change}
\end{equation}
In particular,  $\E(\Delta(\tilde{X_n})) >0$ for $\alpha \in (2/3,1)$.
\end{lem}
\begin{proof}
 Observe that the expected change is given by 
\[
 \alpha - \sum_{i\ge1}i p_i = 1- \sum_{i \ge 1}(i+1)p_i
\]
where $p_i$ is given by \cref{eq: super_pi} and the equality follows from the fact that $\sum_{i \ge 1}p_i = 1-\alpha$. Now from \cref{eq: super_pi},
\begin{equation}
 \sum_{i \ge 1}(i+1)p_i = \sum_{i \ge 1} 2 Cat(i-1)\left(\frac{2/\alpha-2}{4}\right)^i((3\alpha-2)i+1) \label{eq:ec1}
\end{equation}
 where $Cat(n) = \frac{1}{n+1}\dbinom{2n}{n}$ is the $n$th catalan number. The sum in the right hand side of \cref{eq:ec1} can be easily computed using generating functions of catalan numbers. We leave this last step to the reader.
\end{proof}

\begin{lem}\label{lem:a.s.boundary}
 There exists a constant $c>0$ such that almost surely 
\begin{equation}
 c<\liminf \frac{X_n}{n} \le \limsup \frac{X_n}{n} \le 1 \label{eq:bound_hyper}
\end{equation}
\end{lem}
\begin{proof}
 $\limsup X_n/n \le 1 $ follows trivially because $\Delta X_n \le 1$. Since the steps in $\Delta(\tilde{X_i})$ are i.i.d. with finite mean, strong law of large numbers imply that $\tilde{X_n}/n \to \sqrt{3\alpha - 2}{\sqrt{\alpha}}>0$ almost surely as $n \to \infty$. The required lower bound now follows from the fact that $X_n > \tilde{X_n}$ by definition.
\end{proof}
Recall that step $\tau_r$ in the peeling algorithm marks the step when the hull of the ball of radius $r$ is revealed. We now state some estimates on $\tau_r$. The first part of the following \cref{lem:time_estimate} is essentially rephrasing Lemma 4.2 of \cite{UIPT2}. Further, we remark that \cref{lem:time_estimate} is valid for any $\alpha \in [0,1)$ and we shall use it again when dealing with the subcritical case in \cref{sec:subcritical}.

\begin{lem}\label{lem:time_estimate}
For any $r\ge 0$, 
\begin{mylist}
\item There exists some constants $A>1$ and $A' > 0$ such that for any integer $n \ge 1$, 
\begin{equation}
 \P(\Delta \tau_r >  An | |X_{\tau_r}| = n) < \exp(-A'n)
\end{equation}
\item For any integer $k\ge 1$ and integers $1 \le l \le l'$
$$\P(\Delta \tau_r > k | |X_{\tau_r}| = l) \le \P(\Delta \tau_r >  k | |X_{\tau_r}| = l').$$
\end{mylist}
\end{lem}
\begin{proof}
The number of steps required for a vertex on $\partial B_r$ to go inside the revealed map is a geometric random variable (we wait till a step of the form $(L,i)$ occurs for some $i \ge 1$.) Thus $\Delta \tau_r$ is a sum of at most $n$ i.i.d. geometric variables. Thus part (i) follows from a suitable large deviations estimate.

An easy coupling argument can be used to prove part (ii). To see this, let us consider two marked contiguous segments $S$ with $l$ vertices and $S'$ with $l'$ vertices on the boundary with the left most vertex being the root vertex. We can now perform the peeling algorithm described in \cref{sec:peel_algo} until all the vertices in $S$ is inside the revealed map. Clearly, if at some step, some vertices of $S$ are still not swallowed by the revealed map, then some vertices of $S'$ are also not swallowed.
\end{proof}

\cref{lem:time_estimate} along with \cref{lem:a.s.boundary} shows that almost surely for some positive constants $a,a'$ and for all but finitely many $r$
\begin{equation}
 a'\tau_{r+1}<X_{\tau_{r+1}}<\Delta(\tau_r) <aX_{\tau_r} < a\tau_r\label{eq:boundary_bound}.
\end{equation}
For the first and last inequality in the above display, we used \cref{lem:a.s.boundary}, for the third inequality, we used \cref{lem:time_estimate} and for the second inequality we observe that the vertices of $\partial B_{r+1}$ are added only one at a time. This in turn shows that
\begin{lem}\label{lem: stop_time}
 There exists constants $1<c<C$ such that almost surely 
\begin{eqnarray}
\liminf c^{-r}\tau_r &=& \infty \nonumber\\
\limsup C^{-r}\tau_r &<& \infty \nonumber
\end{eqnarray}
\end{lem}
Let $V_n$ denote the number of vertices in the revealed map $P_n$ in the $n$th step of the peeling algorithm. Our main goal is to estimate $V_{\tau_r}$ in order to prove \cref{thm:vol_h_triang}. Now suppose $S_n = V_n - X_n$. Then it is easy to see just from the description of the algorithm that $S_n$ is a sum of $n$ i.i.d. random variables each of which is distributed as $Y+I_{Y+1}$ where 
$Y = -\Delta \tilde{X_n}\mathbbm{1}_{\Delta \tilde{X_n} \neq 1}$
and $I_{Y+1}$ is distributed as a the number of internal vertices of a free triangulation of a $(Y+1)$-gon with parameter $\alpha \beta$. Notice that this definition makes sense for all values of $\alpha$, not for just the supercritical regime.\ However in the supercritical regime, exponential tail of $Y$ entails that $Y$ has finite expectation. Further conditioned on $Y$, the expectation of $I_{Y+1}$ is $4\theta Y/(1-6\theta) + O(1)$ via \cref{prop:free_expectation} where $\theta$ is given by the relation $\theta(1-2\theta)^2 = \alpha\beta$. Thus $Y+I_{Y+1}$ has finite expectation. 
\begin{proof}[Proof of Theorem \ref{thm:vol_h_triang}]
Recall that $|B_r|= V_{\tau_r}$. Now since $S_n$ is a sum of i.i.d. random variables with finite mean, $S_n/n$ converges almost surely. This fact along with \eqref{eq:boundary_bound} and \cref{lem: stop_time} completes the proof.
\end{proof}

\subsubsection{Anchored expansion}\label{sec:anchored}
Now we turn to the proof of \cref{thm:anchored}.
Recall that internal boundary of a simply connected sub-map with a simple boundary is the set of vertices and edges in the sub-map which is incident to at least one finite degree face which do not belong to the sub-map. Also recall that for any two vertices on the boundary, distance along the boundary is the number of edges on the boundary between the vertices. Clearly, distance along the boundary is at least the graph distance in the whole map. We show in the following lemma that the graph distance between vertices on the boundary in the whole map is at least linear in the distance between them along the boundary.

\begin{lem}\label{lem:boundary_linear}
 Let $v$ be a vertex at distance $n \ge 1$ along the boundary from the root vertex on a map with law $\H_\alpha$ where $\alpha \in (2/3,1)$. There exists a constant $t(\alpha)>0$ depending only on $\alpha$ such that the probability of the distance between $v$ and the root being smaller than $t(\alpha)n$ is at most $\exp(-cn)$ for some $c>0$.
\end{lem}
\begin{proof}
 Let us assume without loss of generality that $v$ is to the right of the root vertex.
 We use the peeling algorithm described in \cref{sec:peel_algo} and reveal the hulls of radius $r$ for $r\ge1$ around the root vertex. Recall the notations $P_n$ which denotes the revealed map after $n$ peeling steps and $\tau_r$ which denotes the step in which we finish exploring the hull of radius $r$. Now the vertices of the boundary to the right of the root vertex which goes inside the peeled map is entirely determined by the last step and is easily seen to have exponential tail and a finite expectation depending only on $\alpha$. Hence the probability that $v$ is in the hull of radius at most $tn$ around the root vertex is at most the probability of the event that the sum of $tk$ independent variables with finite expectation and exponential tail is larger than $k$. The latter event has probability $\exp(-ck)$ for some $c>0$ if $t$ is small enough (depending only on $\alpha$) by a suitable large deviations estimate. 
\end{proof}

  A connected segment $X$ on the boundary of the map containing the root edge is said to be a $t$-\textbf{bad segment} for some $t>0$ if there exists a simply connected sub-map $Q$ with a simple boundary whose intersection with the boundary of the map is $X$ and the internal boundary has at most $t|X|$ vertices (see \cref{fig:tbad}).

\begin{figure}
 \centering{\includegraphics[scale = 0.7]{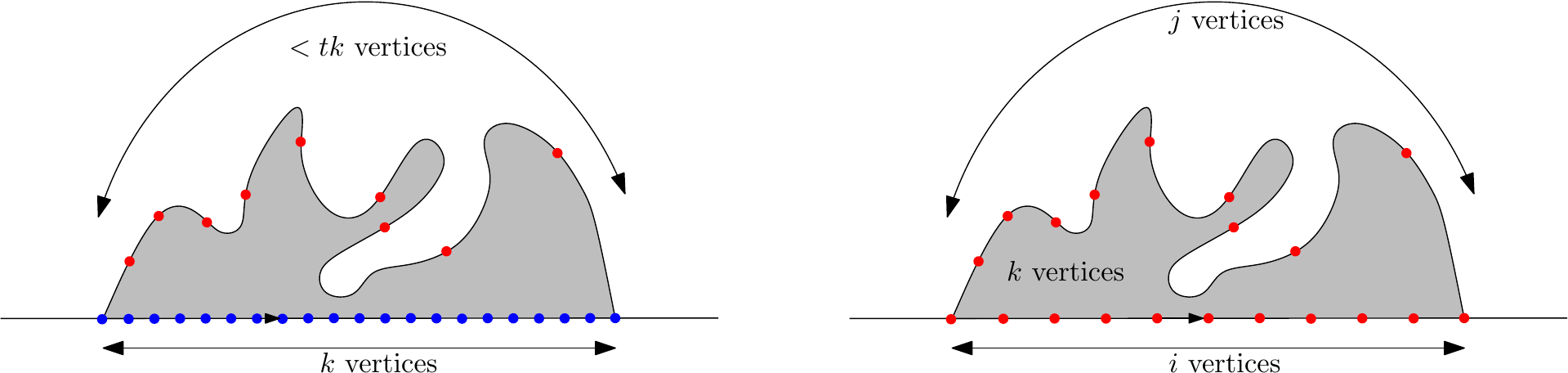}}
\caption{Left: An illustration of a $t$-bad segment. The segment consisting of blue vertices is a $t$-bad segment. The gray area is some fixed finite triangulation. Right: The red vertices form a $(k,i+j)$-separating loop}\label{fig:tbad}
\end{figure}

\begin{lem} \label{lem:segment}
For small enough $t$ (depending on $\alpha$), there exists finitely many $t$-bad segments almost surely. 
\end{lem}
\begin{proof}
 Let us fix a connected segment $X$ of length $k$ containing the root edge. The event that $X$ is $t$-bad is contained in the event that the distance (in the whole map) between the leftmost and the rightmost vertices in $X$ is at most $tk$. If $t>0$ is small enough this event has probability at most $\exp(-ck)$ for some $c>0$ using \cref{lem:boundary_linear} and translation invariance. Since there are at most $k$ connected segments of length  $k$ containing the root, the rest of the proof follows from Borel-Cantelli. 
\end{proof}

 We shall need the following Lemma which essentially follows from Lemma 3.2 of \cite{AR13} and Euler's formula.
\begin{lem}\label{lem:prob}
Fix $\alpha \in  [0,1)$ and $\beta$ is given by \cref{eq:betasub,eq:betasuper}. Let $Q$ be a simply connected triangulation with $k+i+j$ vertices and boundary size $i+j$. Let $T$ be a sample from $\H_\alpha$. The event that $Q$ is a sub-map of $T$ with a marked connected segment containing $i$ vertices on the boundary of $Q$ being mapped to a marked connected segment on the boundary of $T$ with $i$ vertices and no other vertex of $Q$ being mapped to the boundary of $T$ has probability
\[
 \alpha^{k+j}\beta^{i+k-2}
\]
\end{lem}

 We call a simple cycle in a half planar map a \textbf{$(k,l)$ separating loop} if it has $l$ vertices, its intersection with the boundary forms a connected segment containing the root edge and it separates $k$ internal vertices of the map from infinity. 
\begin{lem}\label{lem:expansion}
 Fix $\alpha \in (2/3,1)$. There exists a constant $c(\alpha)$ depending upon $\alpha$ such that $\H_\alpha$-almost surely there are finitely many $(k,l)$-separating loops with $l<c(\alpha)k$.
\end{lem}
\begin{proof}
  Recall that $\phi_{n,m}$ denotes the number of triangulations of an $m$-gon with $n$ internal vertices. From \cref{lem:prob}, \eqref{eq:betasuper} and union bound, the probability that there exists a $(k,l)$-separating loop with $i$ vertices on the boundary of the map is at most
\begin{equation}
  i\phi_{k,l}\alpha^{k+j} \left(\frac{\alpha(1-\alpha)}{2}\right)^{k+i-2} \label{eq:expansion1}
\end{equation}
 where the factor $i$ comes from the fact that the root can be any one of the edges of the intersection of the separating loop with the boundary. Let $j=l-i$. Now it is easy to see from \cref{eq:asymp} that
\begin{equation}
 \phi_{k,l} < (27/2)^k9^{l}\left(1+\frac{2l}{3k}\right)^{2l+3k} k^{-5/2} \sqrt{l} \label{eq:expansion2}
\end{equation}
  Combining \eqref{eq:expansion1} and \eqref{eq:expansion2} and summing over $i<l$ where $l<tk$, we get that the probability of existence of a $(k,l)$-separating loop is at most
\begin{multline}
 l^{5/2}k^{-5/2}\cdot \left(\frac{27 \alpha^2(1-\alpha)}{4}\right)^k \cdot \left(\frac{9\alpha(1-\alpha)}{2}\right)^l \cdot \left(1+\frac{2t}{3}\right)^{(2t+3)k} \cdot \left( \frac{2}{1-\alpha}\right)^{tk} \\<\exp(-ck) \label{eq:expansion3}
\end{multline}
for some constant $c>0$ if $t$ is small enough. To see this, observe that $\alpha^2(1-\alpha) < 4/27$ and $\alpha(1-\alpha) <2/9$ if $\alpha \in (2/3,1)$. The sum of the bound in \eqref{eq:expansion3} over $k>l/t$ and then over $l$ is finite. The rest of the proof follows from Borel-Cantelli.	
\end{proof}
For $S \subset V(G)$, recall that the notation $|S|_E$ denotes the sum of the degrees of the vertices in $S$ and $\partial_ES$ denotes the number of edges which are incident to one vertex in $S$ and another in $G \setminus S$.
\begin{proof}[Proof of \cref{thm:anchored}]
 Consider a connected set of vertices $S$ containing the root vertex such that $|S|_E >n$ and suppose $\partial_E S <t^2|S|_E$ for some $t>0$. By an abuse of notation, denote by $S$ the finite map induced by $S$ and without loss of generality assume it contains the root edge. Add to $S$ all the faces which share at least one vertex with $S$ along with the edges and vertices incident to it. Then add all the connected finite components of the complement and call the resulting finite triangulation $\overline{S}$. Note that $\overline{S}$ is simply connected with a simple boundary and $|\overline{S}|>n$. Also, the vertices and edges in the boundary of $\overline{S}$ form a separating loop. Suppose the internal boundary of $\overline{S}$ has $j$ vertices. From the definition of $\partial_ES$: $j<\partial_E S<t^2|S|_E$. Let $i$ be the number of vertices of $S$ on the boundary of the map and suppose $k = |\overline{S}|-i-j$. Now the assumption $\partial_E S <t^2|S|_E$ and Euler's formula for $\overline{S}$ yields
\begin{equation}
  k> \frac{1-5t^2}{6}|S|_E - \frac{2i}{3}\label{eq:anchored1}
\end{equation}
 If $i<t|S|_E$ then $i+j<t Ck$ for some universal constant $C>0$ using \eqref{eq:anchored1} which can occur for finitely many $n$ almost surely via \cref{lem:expansion} if $t$ is small enough. If $i>t|S|_E$ then $j<ti$ and this can occur for finitely many $n$ almost surely via \cref{lem:segment}. 
\end{proof}
 Proofs of \cref{lem:expansion,lem:segment,thm:anchored} in fact says that the probability of the existence of a set with small boundary containing the root vertex is exponentially small. We record it here for future reference.
\begin{prop}
 There exists a $t>0$ depending only upon $\alpha$ such that the probability that there exists a connected set of vertices $S$ containing the root vertex such that $|S|_E >n$ and  $\partial_E S <t|S|_E$ is at most $\exp(-cn)$ for some $c>0$.
\end{prop}

 \subsection{Subcritical}\label{sec:subcritical}
In this section we prove \cref{thm: boundary_ball_tight}.
We shall use the notations of \cref{sec:peel_algo,sec:supercritical} and assume $\alpha \in [0,2/3)$ throughout this section. Further, $B_r$ will denote the hull of the ball of radius $r$ around the root in a map $T$ with law $\H_\alpha$ Recall that in this regime, probability that a peeling step of the form $(L,i)$ or $(R,i)$ occurs for $i \ge k$ is roughly $k^{-1/2}$.

\begin{subsubsection}{Boundary size estimates}\label{sec: boundary_size_estimates}
To understand the boundary sizes, we need to understand the variables $X_{\tau_r}$ for $r \ge 1$. As a warm up we prove 

\begin{prop}\label{prop:cutset}
In a half planar triangulation with law $\H_\alpha$ where $\alpha \in [0,2/3)$,
there exists infinitely many cutsets each of which consists of a single edge almost surely.
\end{prop}

\begin{proof}
Notice that $X_n \le n+2$ for all $n \in \N$ just by its definition. If the event $\cup_{j >2k+2}\{(R,j)\}$ occur at step $2k$ and the event $\cup_{j>2k+2}\{(L,j)\}$ occur at step $2k+1$, then $X_{2k+2} =2$. But this event has probability at least $c/k$ for some $c>0$ and for different $k$'s these events are independent by the domain Markov property. The proof follows by Borel-Cantelli.
\end{proof}

\cref{prop:cutset} and the Nash-Williams criterion for recurrence (see \cite{MCMT}, Proposition 9.15) immediately implies

\begin{prop}\label{prop:recurrent}
 Simple random walk on a half planar triangulation with law $\H_\alpha$ is recurrent almost surely for $\alpha\in [0,2/3)$.
\end{prop}

 We know via \cref{prop:algorithm} that $X_{\tau_r}$ is an irreducible aperiodic Markov chain with state space $\N \setminus \{0,1\}$. We now show that $\{X_{\tau_r}\}_{r\ge 1}$ is a tight sequence with exponential tail. Suppose $N_{k}(r) $ for $k\ge 0$ denote the number of vertices in the internal boundary of $P_{\tau_r+k}$ which do not belong to $\partial B_r$.


\begin{lem}\label{lem:step}
For any $r\ge 0,k\ge 1$,$n \ge 1$
\[
\P(N_k(r) > n| X_{\tau_r}) <\exp(-Bn)
\]
for some positive constant $B$ which do not depend upon $r,k$ or $n$. In particular, this bound is independent of the conditioning on $X_{\tau_r	}$.
\end{lem}
\begin{proof}
First fix an $n_0$ large enough such that
\begin{equation}
 \alpha - \frac{1}{2}\sum_{i=1}^{\lfloor n_0/2 \rfloor}ip_i < -\varepsilon\nonumber
 \end{equation}
 for some $\varepsilon > 0$ where $p_i$ is given by \cref{eq: sub_pi} (observe that such a choice of $n_0$ exists due to the heavy tail of $p_i$.) The above choice of $n_0$ depends only on $\alpha$.
Now choose an integer $n>n_0$.
Let $\Delta N_k(r) : = N_{k+1}(r) - N_k(r)$ for $ k \ge 1$.
Observe that $N_k(r)$ increases by at most $1$ in any step because of the evolution of $X_k$ and $N_0(r) = 0$. This has several implications. Firstly, this implies that it is enough to consider $k>n$ or otherwise the requested probability is $0$. Secondly, if $N_k(r) > n$, then for some integer $1 \le j \le k$, $N_j(r)$ is equal to $n_0$. Let $M = \max\{1 \le j \le k: N_j(r) = n_0\}$. Finally, we must have $M\le k-n+n_0$. Now note that 
\begin{equation}
\P(N_k(r) > n,M=j)  < \P(N_i(r) \ge n_0 \text{ for all } j\le i \le k)\label{eq:step1}
\end{equation}
Now for any $i > j$, conditioned on $N_i(r) \ge n_0$, there are at least $n_0/2$ vertices of the internal boundary of $P_{\tau_r+k}$ which do not belong to $\partial B_r$ either to the left or right of the edge we perform the $(i+1)$th peeling step because of the way the exploration process evolves. 
Hence it is clear that conditioned on $N_i(r) \ge  n_0$, $\Delta(N_i(r))$ is dominated by a variable $D$ with $\E(D) < -\varepsilon$ because of the choice of $n_0$. Thus, 
\begin{equation}
\P( N_i(r) \ge  n_0 \text{ for all } j\le i \le k)<\P\left(\sum_{i=1}^{k-j}D_i > 0\right) < \gamma^{k-j}\label{eq:step2}
\end{equation}
for some $0<\gamma <1$ depending only on $n_0$ where $\{D_i\}_{i\ge 1}$ are i.i.d. copies of $D$ and the last inequality of \eqref{eq:step2} follows from suitable large deviations estimate. Now using \eqref{eq:step1} and \eqref{eq:step2},
\begin{align}
\P(N_k(r) > n | X_{\tau_r}) =\sum_{j=1}^{k-n+n_0}\P(N_k(r) > n,M=j) < \sum_{j=1}^{k-n+n_0}\gamma^{k-j} < \exp(-Bn)
\end{align}
for some $B>0$ for large enough $n$. Decrease $B$ suitably so that the requested bound is true even for smaller values of $n$.
\end{proof}

We remarked before that \cref{lem:time_estimate} is true for any value of $\alpha$. We shall now use this fact and induction to prove \cref{thm: boundary_ball_tight}.
\begin{proof}[Proof of \cref{thm: boundary_ball_tight}]
First, get hold of the constants $A>1,A'>0,B>0$ such that \cref{lem:time_estimate}, part (i) and \cref{lem:step} are true for $n\ge 1$. Fix a $C$ such that $0<C<B$. Then choose a large $N$ to ensure that for all $n>N$,
$$\max\{\exp(-A'n^2),An^2\exp(-Bn), \exp(-Cn^2)\}<\frac13\exp(-Cn).$$
We shall prove that for all $n>N$, the Theorem is true for the above choice of $C>0$ by induction on $r$. Note that for $r=0$, the Theorem is true trivially since $X_{\tau_0} = X_0 = 1$. Now assume, the Theorem is true for $r'=r-1$ for any $n>N$ for above choice of $C,N$. Now recall the notation $N_j$ from \cref{lem:step} and observe that $N_j = X_j$ for $j \ge \Delta \tau_r$. Clearly, for $n>N$, using \cref{lem:step}
\begin{equation}
\P(X_{\tau_r} >n, \Delta \tau_{r-1} = j | X_{\tau_{r-1}}) < \P(N_j > n | X_{\tau_{r-1}}) < \exp(-Bn)\label{eq:bball1}
\end{equation}
Now for any choice of $n>N$, using \eqref{eq:bball1},
\begin{align}
\P(X_{\tau_r} > n) & < \P(X_{\tau_{r-1}} > n^2) + \P(\Delta \tau_{r-1} > An^2|  X_{\tau_{r-1}} \le n^2) +\sum_{j=1}^{An^2}\exp(-Bn)\nonumber\\
& <\exp(-Cn^2) + \exp(-A'n^2) + An^2\exp(-Bn)\label{eq:bball2}\\
& < \exp(-Cn) \label{eq:bball3}
\end{align}
where \eqref{eq:bball2} follows from induction step, \cref{lem:step,lem:time_estimate}. Also, \eqref{eq:bball3} follows from the choice of $N$. The proof is completed by induction.
\end{proof}

\end{subsubsection}
\begin{subsubsection}{Hull Volumes}\label{sec: hull_volume}

First, we wish to estimate the growth rate of $\tau_r$. Note that conditioned on $X_{\tau_r}$ the distribution of $\Delta(\tau_r)$ depends only on $X_{\tau_r}$ and not $r$. It is easy to see that $Z_r := (X_{\tau_r}, \Delta \tau_r)$ is an irreducible aperiodic Markov chain. Using \cref{thm: boundary_ball_tight,lem:time_estimate} it is not difficult to see that the sequence $\{Z_r\}_{r\ge 1}$ forms a tight sequence. Hence, $Z_r$ has a stationary probability distribution. Let us denote the marginal of the second coordinate of this stationary distribution by $\pi$. It is also easy to see using \cref{thm: boundary_ball_tight,lem:time_estimate} that $\pi$ has exponential tail and hence finite expectation. If we start the Markov chain $\{Z_r\}_{r\ge 1}$ from stationarity, ergodic theorem gives us that $\tau_r/r$ converges almost surely to $\sum_{i\ge 0}i\pi(i)$. However if we start the Markov chain $\{Z_r\}_{r\ge 1}$ from any fixed number, the resulting measure is absolutely continuous with respect to the corresponding chain starting from stationarity. This argument proves

\begin{lem}\label{lem:time_convergence}
Almost surely,
\[
 \frac{\tau_r}{r} \rightarrow \sum_{i\ge 0}i\pi(i)
\]

\end{lem}

Recall the notation $Y,Z,\{S_n\}_{n\ge 1}$ from \cref{sec:supercritical}. Recall that the volume of the triangulation revealed at the $n$-th step of peeling is given by $V_n = S_n+X_n$. We wish to estimate $V_{\tau_r} = |B_r|$. Recall that $S_n$ is a sum of $n$ i.i.d. copies of $W$ where $W =Y+I_{Y+1}$. From \cref{lem:tail} part (i), we conclude $\P(W>x) \sim c_\alpha x^{-1/2}$ as $x \to \infty$ for the constant $c_\alpha$ given by \eqref{eq:const}.

%
%
%
%


\begin{lem}\label{lem:convergence_nthstep}
For some sequence of real numbers $a_n$ and $b_n$
 \begin{equation}
  \frac{V_n - b_n}{a_n} \Rightarrow S
 \end{equation}
where $S$ follows a stable distribution of type $1/2$. Also  $a_n \sim c_\alpha^2n^2$ and $b_n \sim c_\alpha^2 n^2$ where $c_\alpha$ is given by \eqref{eq:const}.
\end{lem}

\begin{proof}
Note that since $X_n \le n$ and since $V_n =S_n+X_n$, it is enough to prove the result with $V_n$ replaced by $S_n$. Since $S_n$ is a sum of an i.i.d. sequence distributed as $W$, we apply \cref{thm: stable}. Recall from \cref{thm: stable}, the centering sequence \mbox{$a_n = \inf\{t: \P(W > t) \le 1/n\}$}. Recall that we also obtained the tail estimate of $W$, \mbox{$\P(W>x) \sim c_\alpha x^{-1/2}$}. It is easy to see from this tail estimate of $W$ that $a_n \sim c_\alpha^2n^2 $.
The asymptotics of $b_n$ is provided in \cref{lem:tail} part (ii).
\end{proof}

We need one final lemma before we prove \cref{thm:hull_convergence}. Recall the distribution $\pi$ from \cref{lem:time_convergence}. 
\begin{lem}\label{lem:finite_hull}
$V_{\tau_{r}}/V_{\lfloor \gamma r \rfloor}$ converges 
in probability to $1$ where $\gamma= \sum_{i\ge0}i\pi(i)$.
\end{lem}
\begin{proof}
 Observe that it is enough to prove $S_{\tau_r}/S_{ \lfloor \gamma r \rfloor}$ converges to $1$ in probability. Notice that since $S_r$ is nondecreasing in $r$,  for any $\eta>0$ and $\varepsilon >0$, we have
\begin{multline}
\P(|S_{\tau_r}/S_{\lfloor \gamma r \rfloor} -1| > \eta , (1-\varepsilon)\gamma r<\tau_r< (1+\varepsilon) \gamma r) \\<
\P\left(\frac{S_{\lfloor (1+\varepsilon)\gamma r \rfloor} - S_{\lfloor (1-\varepsilon)\gamma r \rfloor} }{S_{\lfloor \gamma r \rfloor}} > \eta, (1-\varepsilon)\gamma r<\tau_r< (1+\varepsilon) \gamma r\right)\label{eq:finite_hull1}
\end{multline}
 Recall that a stable law is absolutely continuous (see \cite{fellerII}, Chapter VI.1, Lemma 1) and hence via \cref{lem:convergence_nthstep} we can conclude both $\{S_r/r^2\}_{r \ge 1}$ and $\{r^2/S_r\}_{r \ge 1}$ form a tight sequence in $r$. Further notice that $\tau_r /\gamma r \to 1$ almost surely via \cref{lem:time_convergence}. Combining all these pieces, it is easy to see that for any $\eta>0$, there exists an $\varepsilon>0$ such that the right hand side of \cref{eq:finite_hull1} can be made smaller than any prescribed $\delta>0$ for large enough $r$. The details are left to the reader.
\end{proof}

\begin{proof}[Proof of Theorem \ref{thm:hull_convergence}]
Notice $V_{\tau_r} = |{B_r}|$. Also observe
\begin{equation}
 \frac{V_{\tau_r} - b_{\lfloor \gamma r \rfloor}}{a_{\lfloor \gamma r \rfloor}} = \frac{V_{\lfloor \gamma r \rfloor} - b_{\lfloor \gamma r \rfloor}}{a_{\lfloor \gamma r \rfloor}} + \frac{V_{\lfloor \gamma r \rfloor}}{a_{\lfloor \gamma r \rfloor}}\left( \frac{V_{\tau_r}}{V_{\lfloor \gamma r \rfloor}} -1\right)\label{eq:hull1}
\end{equation}
The first term of \cref{eq:hull1} converges to a stable random variable of type $1/2$ via \cref{lem:convergence_nthstep}. The second term in \cref{eq:hull1} converges to $0$ in probability via \cref{lem:finite_hull}. The proof follows combining these two facts.
\end{proof}

\end{subsubsection}
 
\section{Percolation} \label{sec:percolation1}
  In this Section, we prove \cref{thm:site,thm:density,thm:ends}.
We will use the peeling procedure and use the notations $P_n,T_n$ introduced in \cref{sec:peeling0}. Along with revealing the face on the edge we peel, we might also reveal the color of the new vertex (if any) revealed.
 It will be useful to consider several boundary conditions, which
specifies the colors of the boundary vertices. If we consider a percolation configuration on the whole graph including the boundary vertices, we say it is a random i.i.d. boundary condition.
\begin{description}
 \item  \textbf{Algorithm:} We start with the root vertex black and every other vertex on the boundary white. At the $n+1$th step, we perform a peeling step at the edge on the boundary of $T_n$ with a black vertex to the right and a white vertex to the left. We stop at the $n$th step if there is no
black vertex left on the boundary of $T_n$.
\end{description}

%

Notice that until we stop in the above algorithm, the boundary condition on $T_n$ remains the same as the initial one. A simple topological argument shows that the event that the above
algorithm stops is the same as the event that the black cluster
containing the root vertex
is finite. Now consider the following variable $B$. If the peeling step is an $\alpha$-step and a black vertex is revealed set $B=1$. If the peeling step is of the form $(R,i)$, set $B= -i$. Otherwise set $B=0$. The following Lemma is a computation which essentially follows from \cref{lem:expected_change}.

\begin{lem}\label{lem:computation}
  Suppose $\alpha \in (2/3,1)$
\begin{equation}
\E_p (B) = \alpha p -\frac{1}{2}(\alpha-\sqrt{\alpha}\sqrt{3\alpha-2})
\end{equation}
In particular, $\E_p(B)>0$ if and only if $p > 1/2 (1-\sqrt{3-2/\alpha})$.
\end{lem}
 The following proof is an imitation of the idea of \cite{AC13}. We add it for completeness. Recall the notation $T_n$ from \cref{sec:peeling0}. 

\begin{proof}[Proof of \cref{thm:site}( for $p_c$)]
Assume the following boundary condition: the root vertex is black and the rest of the vertices on the boundary are white. Apply the algorithm described above.\
 Start with $B_0=1$ and suppose $B_k$ is the number of black vertices left in the boundary of $T_k$. Clearly $B_{k+1}-B_k$ are i.i.d. with the same distribution as $B$ as long as $B_{k+1}\neq 0$. \cref{lem:computation} shows that $B_k$ eventually goes to $0$ almost surely if and only if $p \le \frac{1}{2}(\alpha-\sqrt{\alpha}\sqrt{3\alpha-2})$. Modifying the proof to a random i.i.d. boundary is an easy exercise of imitating Proposition 9 of \cite{AC13} and is left to the reader. The almost sure existence of a black cluster if $p > \frac{1}{2}(\alpha-\sqrt{\alpha}\sqrt{3\alpha-2})$ follows from ergodicity of the map with respect to translation of the root (see \cite{AR13}, Proposition 1.3). 
\end{proof}


\begin{corollary}\label{cor:p_uupper}
With random i.i.d. boundary condition, $\H_\alpha$-almost surely,
\begin{equation}
p_u \le 1/2(1+\sqrt{3-2/\alpha})
\end{equation}
\end{corollary}
\begin{proof}
Assume $p \ge
1/2(1+\sqrt{3-2/\alpha}) $. Consider the event $\mathcal E$ that there are two infinite black clusters. Then
one of the components of the complement of one of them must be
infinite. Then the vertices in this component which connect to the
infinite black cluster must be white. This means that there is also an
infinite white cluster since the map is locally finite and one ended almost
surely. Since white clusters are finite almost surely in the given regime of $p$ (using \cref{thm:site} (i) and symmetry), $\mathcal E$ has probability $0$. 
\end{proof}

\begin{figure}[t]
\centering{\includegraphics[scale=1]{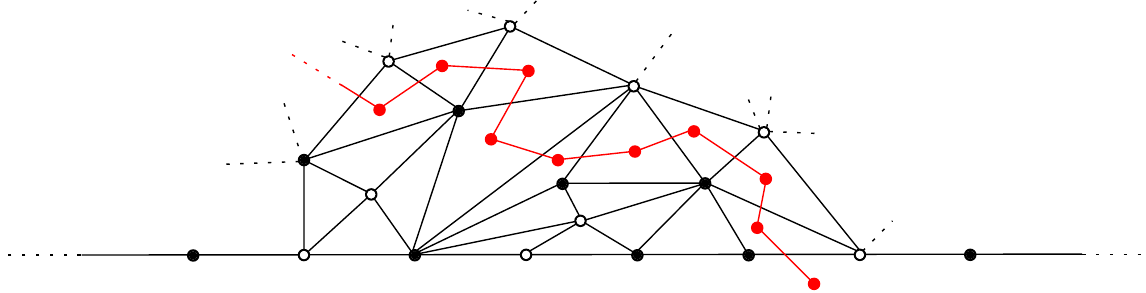} }
\caption{An illustration of an interface between a black and a white cluster for percolation in a half planar triangulation.}\label{fig:interface}
\end{figure}

One can define an interface between the black and white
clusters of a percolation configuration. An interface is a
well defined path in the dual configuration which separates the white
and black clusters (see \cref{fig:interface}). Recall that in the dual configuration of a half planar map, we break up the vertex corresponding to the infinite face into an infinite number of vertices one corresponding to each boundary edge. We are interested in the interfaces which cross a boundary edge. These are the interfaces which start on
those edges on the boundary which are incident to a black vertex and a white vertex. Interfaces mark the boundary between a white and a
black cluster on both its side. An interface might be finite or
infinite. Finite interface separate finite clusters from infinity
while infinite interfaces correspond to an infinite black cluster on
one side and an infinite white cluster on the other. So in particular,
if $p \in
[0,p_c] \cup [p_u,1]$ in a supercritical half planar triangulation, every
interface is finite almost surely.

In the following exploration procedure
the vertices whose colors have not
been revealed yet will be called \textbf{ free} vertices.
\begin{description}
 \item    \textbf{Algorithm 2:} We start with the root vertex colored white, the vertex incident to the right of the root edge colored black and every other vertex on the boundary free. Now we start performing peeling on the root edge.

Suppose after $n$ steps of peeling, the boundary of $T_n$ consists of free vertices except for a finite contiguous white segment followed by a finite contiguous black segment to the right of the white segment. We now peel on the unique boundary edge of $T_n$ connecting the black and white segments. If after a peeling step the third vertex of the face revealed is free, we reveal its
color. If the triangle
revealed {\em swallows} all the black vertices to the right (resp. white
vertices to the left) and the revealed third free vertex is white
(resp. black), then we reveal the colors of the vertices along the boundary to the
right (resp. left) of revealed third vertex until we find a black (resp. white)
vertex. Notice that after such a step, we are again left with a boundary which consists of free vertices except for a finite contiguous white segment followed by a finite contiguous black segment to its right. We can now continue this procedure.  
\end{description}

Let $\mathcal I$ be the event that there is an infinite interface starting from
the root edge.
\begin{lem}\label{lem:interface}
Let $p \in (1/2(1-\sqrt{3-2/\alpha}),1/2(1+\sqrt{3-2/\alpha}))$ and assume random i.i.d. boundary condition. Then
$$\P_p(\mathcal I) > 0.$$
\end{lem}
\begin{proof}
Suppose we are on the event that the root vertex is colored white, the vertex incident to the right of the root edge colored black. Now we perform algorithm $2$. Let
$B_k $ be the size of the black connected segment and $W_k$ be that of
the white connected segment at the $k$th step of the algorithm. Recall the definition of the variable $B$ defined in \cref{lem:computation}. Conditioned on $B_k$, $B_{k+1}$ stochastically
dominates a variable which has the same distribution as
$(B_k+B)^+ + \mathbbm{1}_{\{B_k+B\le 0\}}$ and $B_{k+1}-B_{k}$ are
independent for every $k$. The domination comes from the fact that if
a white segment is {\em swallowed}, we add a geometric $p$ number of
black vertices to $B_k$ which we ignore in the prescribed expression. Now for $p$ in the given range, $\E(B) > 0$,
hence $B_{k}$ forms a random walk with a positive drift. This implies
$B_k \to \infty$ almost surely. Similarly by symmetry, $W_k \to
\infty$ almost surely for $p$ in the given range. All this implies the
event $\{B_k>1,W_k > 1 \text{ for all } k\ge 0\}$ has positive
probability. But $B_k>1$ and $W_k > 1	$ for all $k \ge 0$ implies that
the interface we started with is infinite. This completes the proof.
\end{proof}
   Recall the notations $W_k^{\infty}, B_k^{\infty}$ the number of black and white infinite clusters respectively which has least one vertex on the boundary within distance $k$  along the boundary from the root vertex.
\begin{proof}[Proof of \cref{thm:site} (for $p_u$) and \cref{thm:density}]
Fix a number $p$ in the following range: $p \in (1/2(1-\sqrt{3-2/\alpha}),1/2(1+\sqrt{3-2/\alpha}))$. Let $E_k$ be the number of edges within distance $k$ from the root edge along the boundary such that there is
an infinite interface starting from that edge. Now note that the measure $\P_p$
is ergodic with respect to translation of the root (follows easily from \cite{AR13} Proposition 1.3). Hence Birkhoff's
ergodic
theorem implies that almost surely,
\begin{equation}
  \frac{E_k}{k} \to \P_p(\mathcal I).
\end{equation}
 Note that $W_k^\infty+B_k^\infty =
E_k+1$ and also $|W_k^\infty - B_k^\infty|\le 1$. Hence, 
\begin{equation}
W_k^\infty/k\to \rho \qquad \text{ and } \qquad B_k^\infty/k \to \rho \nonumber
\end{equation}
where $\rho = \P_p(\mathcal I)/2 > 0$ from \cref{lem:interface}. This
proves \cref{thm:density} as well as shows that $p_u \ge
1/2(1+\sqrt{3-2/\alpha})$. 
\end{proof}

Now we turn to the proof of \cref{thm:ends}. We will need the following technical Lemma which can be easily shown using optional stopping Theorem. For details, we refer the reader to \cite{gallager1996discrete} Corollary 9.4.1 and Exercise 9.13. 
\begin{lem}\label{lem:LD}
Let $X_1,X_2,\ldots$ be an i.i.d. sequence of random variables such that $E(X_1) > 0$ and $\E(\exp(\lambda X_1))$ exists for values of $\lambda$ in a neighbourhood around $0$. Let $S_n = \sum_{i=1}^n X_i$. Then for any $k>0$ there exists some constant $c>0$ such that 
\[
\P(\cup_{n\ge 1}\{S_n \le -k\}) < \exp(-ck)
\]
\end{lem}

 \begin{lem}\label{lem:one_end}
  Fix $p \in (p_c,p_u)$ and assume random i.i.d. boundary condition. The $\P_p$ probability that the root vertex is contained in an infinite black cluster with one end or an infinite white cluster with one end is $0$.
 \end{lem}

\begin{proof}
Suppose without loss of generality the color of the root vertex is black and we shall prove that the probability that this vertex is contained in an infinite black cluster with one end is $0$. Reveal vertices to the left and right of this vertex along the boundary until we find a white vertex on both sides. In the exploration we describe now, there will be a contiguous finite white segment followed by a contiguous finite black segment followed by a contiguous finite white segment on the boundary and the rest of the vertices on the boundary are free. We shall peel alternately at the two edges connecting the black and the white segments to the left and to the right. If at any step we swallow all the black vertices we stop. If we swallow all the white vertices to the left (resp. to the right), we reveal black vertices to the left (resp. to the right) along the boundary until we find a white vertex. Consider the sequence of maps $T_n$. Define the root edge of this map to be the same root edge as in the previous step if it has not been swallowed in that step. If it is swallowed, define the edge in the middle of the black segment in the boundary of $T_n$ oriented from left to right to be the new root edge.
\begin{lem}\label{lem:finite_swallow}
The root edge is swallowed finitely many times almost surely in the above described exploration.
\end{lem}
\begin{proof}
Notice that on the event we stop the exploration, the Lemma is true by definition.
 Let $L_n$ (resp. $R_n$) be the distance between the root vertex and edge to the left (resp. right) on which we perform the $n$th peeling step in the above described exploration and let $B_n = L_n+R_n$ be the length of the black segment. Clearly, the sequence $\{\Delta B_n\}_{n\ge 1}$ is an i.i.d. sequence of variables with each of which is distributed as $B$. Recall that $B$ has positive expectation in the given regime of $p$ (using \cref{lem:computation}). Hence using standard large deviation estimates, on the event that we do not stop the exploration, the probability of $B_n \le tn$ for small enough $t$ has probability at most $\exp(-cn)$ for some constant $c>0$. 

Now consider the event $\mathcal E_n$ that the root edge is swallowed in the $n$th step and is swallowed again in some step after the $n$th step. On the event $B_n >tn$ if the root edge is swallowed in the $n$th step, then by description of the exploration both $L_n$ and $R_n$ are at least $tn/2-1$. If the root edge is swallowed again, then either $\{L_k\}_{k\ge n}$ or $\{R_k\}_{k \ge n}$ has to reach $0$ starting from at least $tn/2-1$. This event has probability at most $\exp(-c'n)$ for some $c'>0$ via \cref{lem:LD} since $L_n$ as well as $R_n$ has i.i.d. increments with positive expectation in every alternate step until the root edge is swallowed. Combining the pieces, we see that $\mathcal E_n$ has probability at most $\exp(-c''n)$ for some $c''>0$ which means $\mathcal E_n$ occurs for finitely many $n$ by Borel-Cantelli lemma. This completes the proof.
\end{proof}

\begin{figure}[t]
\centering{\includegraphics[scale=0.7]{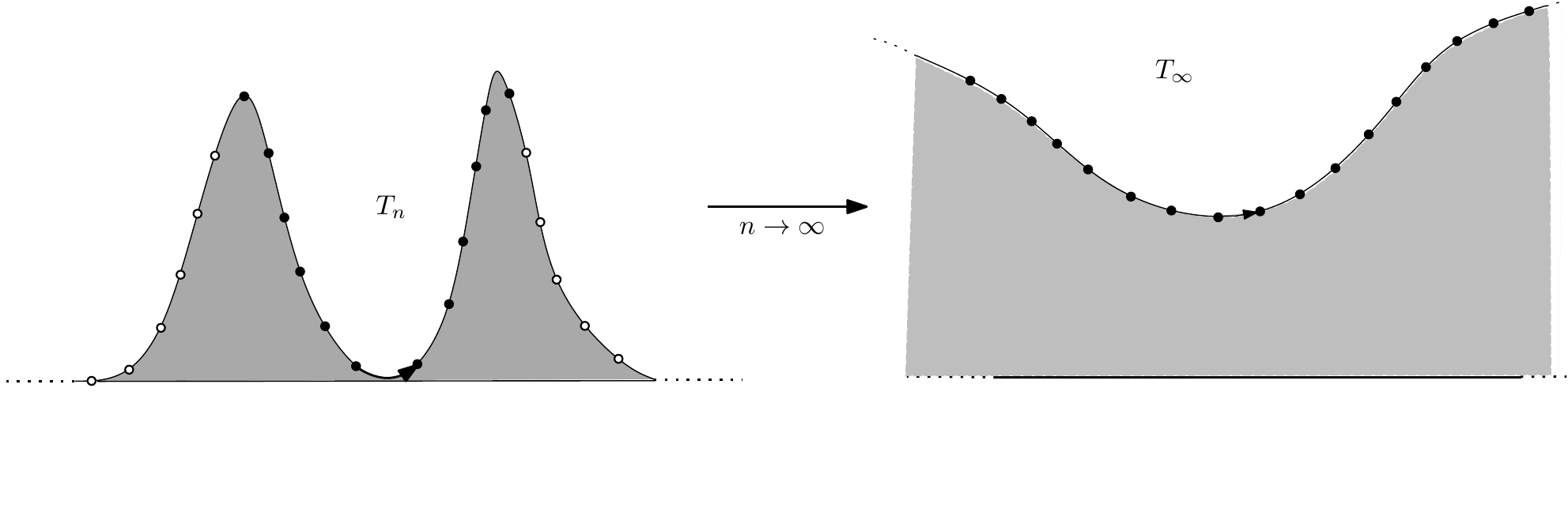} }
\caption{An illustration of the proof of \cref{lem:one_end}. The gray area to the left denotes the revealed part in the exploration.}\label{fig:one_end_lemma}
\end{figure}

Let $T$ be a map with law $\P_p$ and we perform the above exploration.
Let $B_n$ be the number of black vertices on the boundary of $T_n$. On the event that $B_n \to \infty$, $T_n$ converges almost surely to a submap $T_\infty$ of $T$ since the root edge is swallowed finitely often almost surely via \cref{lem:finite_swallow} (see \cref{fig:one_end_lemma}). However, on the event $B_n \to \infty$, from the domain Markov property, $T_\infty$ has law $\P_p$ with all boundary vertices black and the rest of the vertices free. This map almost surely contains an infinite white cluster for the given range of $p$ via \cref{thm:site}. To see this notice that we can perform peeling until we expose the hull of radius $1$ around the boundary (that is all the faces which are incident to at least one vertex of the boundary along with the finite components of their complement.) The complement of this revealed map is another independent percolation configuration on  a half planar triangulation with law $\H_\alpha$ and random i.i.d. boundary condition. So we can apply \cref{thm:site}. However the presence of an infinite white cluster in $T_\infty$ means that the cluster containing the root has at least two ends almost surely on the event that the cluster is infinite. The rest of the details are left to the reader. 
\end{proof}

 \begin{corollary}\label{cor:isolated_end}
Every infinite cluster do not contain an isolated end almost surely.
\end{corollary}
\begin{proof}
We prove the corollary for an infinite cluster containing the root vertex. The proof for any infinite cluster is an easy exercise using the domain Markov property, and is left to the reader. Let $B_r$ be the hull of the ball of radius $r$ from the root of a map with law $\H_\alpha$. Let $T$ be a half planar triangulation with law $\P_p$. Suppose with positive probability there is an infinite cluster in $T$ containing the root vertex which has an isolated end. This implies that with positive probability there exists an $r$ such that $T \setminus B_r$ has an infinite cluster incident to the boundary with one end. This is a contradiction because of \cref{lem:one_end} and domain Markov property.   
\end{proof}

\begin{proof}[Proof of \cref{thm:ends}]
 \cref{cor:isolated_end} shows that each infinite cluster do not contain an isolated end. Since $\END$ is compact, the non-isolated points form a perfect subset via the Cantor Bendixson Theorem. Hence this implies that the set of ends has cardinality of the continuum (see \cite{Kurbook}). 
\end{proof}

\section{Conclusion}\label{sec:open}

\subsection{Spectral dimension for $\alpha \in [0,2/3)$.} Recall that the spectral dimension of a graph $d_s(G)$ is defined as 
\[
  d_s(G) = -2 \lim_{n \to \infty} \frac{\log p^G_{2n}(x,x)}{\log n} 
\]
 where $p^G_{2n}(x,x)$ is the probability that simple random walk on the graph $G$ starting from $x$ returns to $x$ in $2n$ steps. The small cutsets of the subcritical maps give us the hint that these maps fall in the class of strongly recurrent graphs as in \cite{kumagai2008heat} and the spectral dimension should be almost surely $4/3$. Let us now argue briefly why this should be true. We alert the reader is that what follows is not a rigorous proof but just a proof outline which we believe can be made into a complete proof by the diligent reader.

Consider the functions $v,r: \N \to [0,\infty)$ where $v(R) = R^2$ and $r(R)=R$. Then $v$ and $r$ satisfy (1.12) of \cite{kumagai2008heat}. The function $v(R)$ correspond to the volume frowth of the ball of radius $R$ and $r(R)$ correspond to the effective resistance (see \cite{MCMT} for background) between the root and the complement of the ball of radius $R$.
Let $\P$ denote the measure of a simple random walk $X_0,X_1,\ldots$ starting from the root $\rho$ on a map $T$ with law $\H_\alpha$. Let $E_T$ denote the expectation corresponding to the simple random walk measure given an instance of the map $T$ and let $d(.,.)$ denote the graph distance metric on the map $T$. Also suppose $e_R$ denote the smallest time when the simple random walk is not in the ball of radius $R$ around the root.

 We aim to give a sketchy argument to show that assumption 1.2 (1) and (3) of \cite{kumagai2008heat} are satisfied for the above choice of $v$ and $r$. Then it would follow via equation (1.19) of Proposition 1.3 and Theorem 1.5(III) of \cite{kumagai2008heat} that 
\begin{mylist}
\item $\P (M^{-1} < (1+d(\rho,X_n))n^{-2/3} , d(\rho,X_n)n^{-2/3} <M )  \xrightarrow[]{M \to \infty} 1$

 \item $\H_\alpha$- almost surely, \[
       d_s(T):= -2\lim_{n \to \infty}\frac{\log p^T_{2n} (\rho,\rho)}{\log n} = \frac43 
                                   \]

\item$\H_\alpha$- almost surely, \[
       \lim_{R \to \infty} \frac{ \log E_T (e_R)}{\log R} = 3
      \]
\end{mylist}

 To show that assumption 1.2 (1) and (3) of \cite{kumagai2008heat}, it is enough to show that for all large enough $\lambda$, $R\ge 1$ and some constant $c>0$,
\begin{align}
 \P(|Ball_R|/R^2 \in (\lambda^{-1},\lambda)) & > 1-c\lambda^{-1/2} \label{eq:speedsub1}\\
  \P(R_{\text{eff}}(0,T \setminus Ball_R) \ge \lambda^{-1}r) & >1-c\lambda^{-1}\label{eq:speedsub2}
\end{align}
where $Ball_R$ denote the ball of radius $R$ around the root vertex.
Recall the notations and the peeling algorithm to reveal the hull of the ball of radius $r$ in \cref{sec:peel_algo}. Notice that the edges of the triangle revealed (excluding the finite holes) in every step of peeling for all steps from $\tau_r$ to $\tau_{r+1}$ gives a cut-set separating the root from the complement of the ball of radius $R$ if $r <R$. Further these cutsets are disjoint for different $r$'s. Hence using Nash Williams criterion (\cite{MCMT}, Prop 9.15), 
\begin{equation}
   R_{\text{eff}}(0,T \setminus B_R) \ge \sum_{k=1}^R (2\Delta \tau_k)^{-1}\label{eq:speedsub3}
\end{equation}
 The required lower bound for the right hand side of \eqref{eq:speedsub3} follows from the fact that $\Delta \tau_R$ converges to a stationary distribution. To show \eqref{eq:speedsub1}, we need to show that the ball of radius $R$ is roughly $R^2$. An upper bound follows from \cref{thm:hull_convergence} and the term $\lambda^{-1/2}$ is obtained from the tail of the stable $1/2$ variable. The only problem is obtaining the lower bound for the ball volume. But this should not be too difficult to obtain using the ideas of \cite{UIPT2} Section 6. We do not attempt to show this part of the proof in this paper.

\medskip
\subsection{Open questions}

We conclude with several open problems for possible future research. In \cref{thm:vol_h_triang}, it is shown that the volume growth is exponential. A natural question is: what is the exact rate of growth of the volume? We expect similar behaviour as exhibited by a supercritical Galton-Watson tree.
\begin{question}
Suppose $T$ is a map with law $\H_\alpha$ where $\alpha \in (2/3,1)$. Show that almost surely,
\[
\frac{\log |B_r(T)|}{r} \to c
\] 
for some constant $c$ depending only on $\alpha$. Show further that $|B_r(T)|/c^r$ converges to some non-degenerate random variable.
\end{question}
In \cref{thm:ends} it is shown that the supercritical percolation clusters in the regime $p \in (p_c,p_u)$ have uncountably any ends. It would be interesting to know how a supercritical percolation cluster behave. 
\begin{question}
Fix $\alpha \in (2/3,1)$ and $p \in (p_c,p_u)$. Does the supercritical percolation cluster have exponential volume growth? Anchored expansion? Is the simple random walk on it transient? Has positive speed? 
\end{question}
The key to understand the supercritical cluster in this regime is to understand if the supercritical clusters have long thin cutsets which kills anchored expansion. 

\begin{appendix}\label{sec:appendix}

\section{Proof of \cref{lem:tail}}\label{sec:tailproof} 

Recall that $I_{m}(q)$ is the number of internal vertices of a free triangulation of a $m$-gon with parameter $q$ and recall the variable $\theta$ used in \cref{sec:counting_formulas} where $q = \theta(1-2\theta)^2$.

%
%
%
%

\begin{proof}[Proof of \cref{lem:tail} part (i)]
 Without loss of generality assume $x$ is an integer. Let $d_\theta = \frac{4\theta}{(1-6\theta)}$. For simplicity of notation let $I_m(q) = I_m$. Notice that conditioned on $Y=k$, expectation of $I_{Y+1}$ is $d_\theta k + O(1)$ as $k \to \infty$. We want,
\begin{align}
 \P(Y+I_{Y+1} > x)& =\sum_{k\ge 1} \P(I_{Y+1}>x-k |Y=k) \P(Y=k) \label{eq:tail0}
\end{align}
The trick is to break the sum in \eqref{eq:tail0} into sums over three subsets of indices:
\begin{mylist}
 \item $A_1 = \{1 \le k \le \lfloor x/(1+d_\theta) - x^{3/4} \rfloor\}$
\item  $A_3 = \{k > \lfloor x/(1+d_\theta) + x^{3/4} \rfloor\}$
\item $A_2 = \N \setminus (A_1 \cup A_3	)$
\end{mylist}
The sum over $A_2$ is $O(x^{-3/4})$ by bounding $\P(I_{Y+1}>x-k|Y=k)$ by $1$ and using $\P(Y=k) \sim ck^{-3/2}$. Now note 
\begin{align}
 & \sum_{A_1}\P(I_{Y+1}>x-k |Y=k) \P(Y=k)\label{eq:tail1} \\
< &\sum_{A_1}\P(I_{Y+1} -\E(I_{Y+1}|Y=k)>x^{3/4} + O(1) |Y=k) \P(Y=k)\label{eq:tail2} \\
 < &\sum_{A_1}\frac{Var(I_{Y+1}|Y=k)}{x^{3/2}}\P(Y=k) = O(x^{-1}) \label{eq:tail4}.
\end{align}
where we used \cref{prop:free_expectation} part (i) for \eqref{eq:tail2} and Chebyshev's inequality followed by \cref{prop:free_expectation} part (ii) for \eqref{eq:tail4}. Finally, 	
\begin{align}
&\sum_{k\in A_3} \P(I_{Y+1}>x-k |Y=k) \P(Y=k) \label{eq:tail6}\\
=& \sum_{k\in A_3}\P(Y=k) - \sum_{k\in A_3}\P(I_{Y+1} \le x-k |Y=k) \P(Y=k)\label{eq:tail8}\\
 =& \sum_{k\in A_3}\P(Y=k) - O(x^{-1}) \label{eq:tail3}
\end{align}
where the bound in the second term in the right hand side of \eqref{eq:tail3} follows in the same way as \eqref{eq:tail4} using Chebyshev's inequality and \cref{prop:free_expectation} part (ii) plus the fact that the summands are $0$ when $k>x$. Finally it is easy to verify using \eqref{eq:constant1} and the definition of $d_\theta$ that
 \[                                                                                                                                                                                                                                                                                       \sum_{k\in A_3}\P(Y=k) \sim \frac{(1-3\alpha/2)\sqrt{1-2\theta}}{\sqrt{\pi(1-6\theta)}}x^{-1/2}                                                                                                                                                                                                                                                                                       \]
 
\end{proof}
\medskip
\begin{proof}[Proof of \cref{lem:tail} part (ii)]

Note that
 \begin{equation}
   \E((Y+I_{Y+1}) \mathbbm{1} _ {\{Y+I_{Y+1} <x\}} )= \sum_{k=1}^{x-1} (\P(Y+I_{Y+1}\ge k) - \P(Y+I_{Y+1} \ge x)) \label{eq:tail5}
 \end{equation}
 Now from the asymptotics of part (i),
\begin{equation}
 \sum_{k=1}^{x-1} (\P(Y+I_{Y+1}\ge k) =2c_\alpha\sqrt{x}(1+o(1))
\end{equation}
 and 
\begin{equation}
( x-1)\P(Y+I_{Y+1} \ge x) = c_\alpha \sqrt{x}(1+o(1))
\end{equation}
Hence the result follows.
\end{proof}

\end{appendix}

\bibliographystyle{abbrv}
\bibliography{geometry}

\end{document}